\documentclass[12pt]{article}
\usepackage{amsthm,amsfonts,amsmath,amscd,amssymb,latexsym}
\usepackage{mathrsfs}
\usepackage{epsfig,graphicx,color}
\usepackage[all]{xy}
\usepackage{makeidx}
\usepackage{hyperref}
\usepackage{enumerate}
\usepackage{mathtools}

\title{The Donaldson geometric flow is a local smooth semiflow}

\author{
  Robin~S.~Krom\thanks{Partially supported by the
  Swiss National Science Foundation Grant 200021-127136}
  \\
  ETH Z\"urich
}

\date{29 December 2015}

\newtheorem{PARA}{}[section]
\newtheorem{theorem}[PARA]{Theorem}
\newtheorem{corollary}[PARA]{Corollary}
\newtheorem{lemma}[PARA]{Lemma}
\newtheorem{proposition}[PARA]{Proposition}

\newtheorem{remark}[PARA]{Remark}

\newcommand{\MAT}[1]{\left[\begin{array}{#1}}
\newcommand{\RIX}{\end{array}\right]} %
\newcommand{\p}{\partial}

\newcommand{\C}{{\mathbb C}}

\newcommand{\N}{{\mathbb N}}

\newcommand{\R}{{\mathbb R}}

\newcommand{\cH}{{\mathcal H}}

\newcommand{\cL}{{\mathcal L}}

\newcommand{\cP}{{\mathcal P}}

\newcommand{\sE}{\mathscr{E}}
\newcommand{\sF}{\mathscr{F}}

\newcommand{\sR}{\mathscr{R}}
\newcommand{\sS}{\mathscr{S}}
\newcommand{\sT}{\mathscr{T}}
\newcommand{\sU}{\mathscr{U}}

\newcommand{\sX}{\mathscr{X}}
\newcommand{\sY}{\mathscr{Y}}
\newcommand{\sZ}{\mathscr{Z}}
\newcommand{\Om}{{\Omega}}
\newcommand{\om}{{\omega}}

\renewcommand{\phi}{{\varphi}}

\newcommand{\fhat}{{\widehat{f}}}

\newcommand{\Khat}{{\widehat{K}}}
\newcommand{\rhohat}{{\widehat{\rho}}}
\newcommand{\rhobar}{{\bar{\rho}}}
\newcommand{\etahat}{{\widehat{\eta}}}
\newcommand{\etatilde}{{\widetilde{\eta}}}
\newcommand{\zetahat}{{\widehat{\zeta}}}
\newcommand{\zetatilde}{{\widetilde{\zeta}}}
\newcommand{\rhotilde}{{\widetilde{\rho}}}
\newcommand{\Ktilde}{{\widetilde{K}}}
\newcommand{\Ltilde}{{\widetilde{L}}}

\newcommand{\phihat}{{\widehat{\phi}}}

\newcommand{\tauhat}{{\widehat{\tau}}}

\newcommand{\thetahat}{{\widehat{\theta}}}
\newcommand{\starhat}{{\widehat{*}}}
\newcommand{\sXhat}{{\widehat{\sX}}}
\newcommand{\sYhat}{{\widehat{\sY}}}
\newcommand{\sZhat}{{\widehat{\sZ}}}

\newcommand{\id}{{\mathrm{id}}}

\newcommand{\Vol}{{\mathrm{Vol}}}

\newcommand{\Aut}{{\mathrm{Aut}}}

\newcommand{\Vect}{{\mathrm{Vect}}}

\newcommand{\norm}{{\rm norm}}

\newcommand{\dvol}{{\rm dvol}}

\newcommand{\fp}{{\mathfrak{p}}}

\newcommand{\CP}{{\C\mathrm{P}}}

\newcommand{\inner}[2]{\langle #1, #2\rangle}

\def\NABLA#1{{\mathop{\nabla\kern-.5ex\lower1ex\hbox{$#1$}}}}
\def\Nabla#1{\nabla\kern-.5ex{}_{#1}}

\def\Abs#1{\left|#1\right|}
\def\norm#1{\mathopen\|#1\mathclose\|}

\renewcommand{\p}{{\partial}}

\begin{document}
\maketitle
\begin{abstract}
  We prove the local existence of unique smooth solutions of the Donaldson
  geometric flow on the space of symplectic forms on a closed smooth
  four-manifold, representing a fixed cohomology class.  It is a semiflow on
  the Besov space $B^{1,p}_2(M, \Lambda^2)$ for $p > 4$. The Donaldson
  geometric flow was introduced by Simon Donaldson in~\cite{DON1}.  For a
  detailed exposition see~\cite{KROMSAL}.
\end{abstract}

\tableofcontents


\section{Introduction}\label{sec:intro}

Let $M$ be a smooth closed Riemannian four-manifold. Denote by $g$ the
Riemannian metric, denote by $\dvol\in \Om^4(M)$ the volume form of $g$ and let
$*:\Om^k(M) \to \Om^{n-k}(M)$ be the Hodge $*$-operator associated to the
metric and orientation. Let $\om$ be a symplectic form on $M$ compatible with
the metric and let $\sS_a$ be the space of symplectic forms representing the
cohomology class $a = [w] \in H^2(M, \R)$. This is formally an infinite
dimensional manifold and the tangent space at any element $\rho\in \sS_a$ is
the space of exact two-forms. The Donaldson geometric flow on $\sS_a$ is given
by the evolution equation
\begin{equation}\label{eq:DONFLOW}
  \frac{d}{dt}\rho_t = d *^{\rho_t} d \Theta^{\rho_t},
\end{equation}
where
\begin{equation}\label{eq:theta_u_star}
  \Theta^\rho := *\frac{\rho}{u} -
  \frac12\Abs{\frac{\rho}{u}}^2\rho,\qquad \frac12 \rho\wedge \rho =:
  u \dvol,\qquad *^\rho \lambda := \frac{\rho\wedge *\left( \rho\wedge
    \lambda
  \right)}{u}
\end{equation}
for all one-forms $\lambda\in \Om^1(M)$. The operator $*^\rho$ is the Hodge
star operator for the metric $g^\rho$ which is uniquely determined by the
conditions
$$
\dvol_{g^\rho} = \dvol,\qquad  *^\rho \left( \om -
\frac{\om \wedge \rho}{\dvol_\rho} \rho \right) = \left( \om - \frac{\om \wedge
\rho}{\dvol_\rho} \rho \right)\Leftrightarrow*\om = \om
$$
for all $\om \in \Lambda^2T^*M$. Each $\rho \in \Om^2$ with $\rho^2>0$
determines an inner product $\inner{\cdot}{\cdot}_\rho$ on the space of exact
two-forms defined by
\begin{equation}\label{eq:metric}
  \inner{\rhohat_1}{\rhohat_2}_\rho := \int_M \lambda_1 \wedge *^\rho \lambda_2,
\end{equation}
These inner products determine a metric on the infinite dimensional space
$\sS_a$ called the Donaldson metric.
The Donaldson geometric flow is the negative gradient
flow with respect to the Donaldson metric of the energy functional $\sE:\sS_a
\to \R$ defined by
\begin{equation}\label{eq:energy}
  \sE(\rho) := \int_M \frac{2\Abs{\rho^+}^2}{\Abs{\rho^+}^2 -
  \Abs{\rho^-}^2}, \qquad \rho \in \sS_a.
\end{equation}
The Donaldson flow has a beautiful geometric origin laid out in Donaldson's
paper \cite{DON1}. The key idea is that the space of diffeomorphisms of a
hyperK\"ahler surface has the structure of an infinite dimensional
hyperK\"ahler manifold. The group of symplectomorphisms with respect to a
preferred symplectic structure $\om$ then acts by composition on the right and
this group action is generated by a hyperK\"ahler moment map. In analogy with
the finite dimensional case one then studies the negative gradient flow to the
moment map square functional with respect to the $L^2$-inner product. If we
push the preferred symplectic structure $\om$ forward by the diffeomorpisms of
$M$ to the space of symplectic structures in a fixed cohomology class, we
obtain the Donaldson flow \eqref{eq:DONFLOW}. If we push the moment map square
forward we obtain the energy functional $\sE:\sS_a \to \R$ in
\eqref{eq:energy} and if we push the $L^2$-metric forward we obtain the
Donaldson metric~\eqref{eq:metric}.  The Donaldson flow remains the negative
gradient flow to this energy functional with respect to the Donaldson metric
\eqref{eq:metric} on the space of symplectic structures in a fixed cohomology
class. The Donaldson flow remains well defined for a general symplectic
4-manifold $(M, \om)$ equipped with a compatible Riemannian metric $g$.

The motivation to study the Donaldson flow on the space of symplectic
structures comes from the longstanding open uniqueness problem for symplectic
structures on closed four-manifolds (see \cite{SAL} for an exposition). The
hope is that the Donaldson flow provides a tool to settle this question at
least in some favorable cases, such as the hyperK\"ahler surface and the
complex projective plane $\CP^2$. This hope is strengthened by the observation
that the preferred symplectic structure $\om$ is the unique absolute minimum
and that the Hessian of the energy functional $\sE$ is positive definite at
the absolute minimum (see \cite{KROMSAL}). In the case of $M=\CP^2$ we can
further show that the Fubini-Study form is the only critical point of the
flow. Thus, if longtime existence and convergence of the Donaldson flow can be
established, the Donaldson flow would provide a proof for the uniqueness of
the symplectic structures on $\CP^2$ of a given cohomolgy class up to isotopy.
In the case of the hyperK\"ahler surface Donaldson \cite{DON1} proved that the
higher critical points are not strictly stable. The idea would then be to
perturb the flow near the critical points such that eventually the perturbed
flow converges towards the unique absolute mimimum.

The main results of this paper are regularity for the critical points and the
regularity for the flow (section~\ref{sec:regularity}), short time existence
(section~\ref{sec:ste}) and the semiflow property
(section~\ref{sec:semiflow}). These results lay the foundation for future
studies focusing on longtime existence and the problem of solutions
\emph{escaping to infinity}.  Section~\ref{sec:K} contains the main geometric
ideas, in particular a result by Donaldson \cite{DON2} that shows that the map
$\rho \mapsto \frac{\rho^+}{u}$ from the space of symplectic forms
representing a cohomology class in a fixed affine space in $H^2(M; \R)$ with
dimension equal to $b_2^+$ to the space of self-dual two-forms is locally a
Banach space diffeomorphism. In section~\ref{sec:reducedEvoEq} we then study
the evolution of $\frac{\rho^+}{u}$ if $\rho$ evolves by the Donaldson flow.
This evolution equation is the key to prove regularity of the Donaldson flow.
Sections~\ref{sec:ste} and~\ref{sec:semiflow} are independent of the previous
ones on regularity and can be read separately. In section~\ref{sec:ste} we use a
standard argument involving Banach's fixed point theorem to prove short time
existence in an appropriate Sobolev completion of the space of two-forms. In
section~\ref{sec:semiflow} we prove the semiflow property of the flow on the
related interpolation space. In particular we can use the local stable
manifold theorem to show that there exists an attracting neighborhood around
the absolute minimum in the topology of this space. The appendix deals with
products and compositions of functions in Sobolev spaces.

{\bf Notation and Conventions.}\label{notation}
Let $(M, g)$ be a closed Riemannian oriented four-manifold.  Let $\pi: E
\rightarrow M$ be a natural $k$-dimensional vector bundle over $M$.  We denote
by $W^{\ell,p}(M, E)$ the Sobolev completion of the space of sections
$\Om^0(M, E)$. If the bundle in question is clear from the context we will
just write $W^{\ell,p}$ instead of $W^{\ell,p}(M, E)$.  We will suppress the
constants that solely depend on the parameters $\text{dim}(M), \text{vol}(M),
k, p$ from the notation when we make estimates in Sobolev norms.  We denote by
$\sS$ the space of smooth symplectic structures on $M$ compatible with the
orientation. We write $\sS_a$ for $\rho \in \sS$ that represent a fixed
cohomology class $a\in H^2(M; \R)$. Let $L \in H^2(M; \R)$ be an affine
subspace. Then $\sS_L$ denotes the subset of $\sS$ such that the symplectic
structures represent cohomology classes in $L$. We define
\begin{gather*}
  \sS^{k,p} := \{\rho \in W^{k,p}(M, \Lambda^2) |\, \rho\wedge \rho > 0,\
d\rho = 0\}.
\end{gather*}
A cohomology class or an affine subspace of cohomology classes as
subscript has the same meaning as in the smooth counterpart.


\section{The Map \texorpdfstring{$K$}{K}}\label{sec:K}
In this section we study the map
$$
\sS \to \Om^+: \qquad \rho \to \frac{\rho^+}{u} ,\qquad u:= \frac{\rho\wedge
\rho}{2\dvol}
$$
from the space of symplectic forms to the space of self-dual two-forms.  We
apply a theory developed by Donaldson \cite{DON2} on elliptic problems for
closed two-forms on four dimensional closed manifolds that fullfill a point
wise constraint with `negative tangents'. The main insight is that this map is
a local Banach space diffeomorphism between appropriatly choosen spaces and
the regularity of $\frac{\rho^+}{u}$ determines the regularity of $\rho$. The
precise statement is given in Theorem~\ref{thm:K}.

We will use the construction of a Riemannian metric determined by a
nondegenerate two-form and a volume form explained in~\cite{KROMSAL}. Here is
a brief overview. Fix a nondegenerate two-form $\rho\in \Om^2(M)$ such that
$\rho\wedge \rho >0$. There exists a Riemannian metric $g^\rho$ such that its
volume form agrees with $\dvol$ of $(M,g)$. The associated Hodge star operator
$*^\rho: \Om^1(M) \to \Om^3(M)$ is given by
$$
*^\rho \lambda = \frac{\rho \wedge *\left(\rho\wedge\lambda\right)}{u}.
$$
If $X\in \Vect(M)$ is a vector field then
$$
*^\rho \rho(X,\cdot) = -\rho \wedge g(X,\cdot).
$$
The map
\begin{equation}\label{def:Rrho}
  R^\rho: \Om^2(M) \to \Om^2(M), \qquad R^\rho \om := \om - \frac{\om \wedge
  \rho}{\dvol_\rho}\rho
\end{equation}
is an involution that preserves the exterior product, acts as the identity on
the orthogonal complement of $\rho$ with respect to the exterior product and
it sends $\rho$ to $-\rho$. Moreover it maps $\Om^+$ to $\Om^{+^\rho}$, the
self-dual forms with respect to the metric $g^\rho$.  The Hodge star operator
$*^\rho : \Om^2(M) \to \Om^2(M)$ associated to $g^\rho$ is given by
$$
*^\rho \om = R^\rho * R^\rho \om.
$$
Let $\om\in \Om^2(M)$ be a self-dual two-form and let $J: TM\to TM$ be an
almost complex structure such that $g = \om(\cdot, J\cdot)$. We define the
almost complex structure $J^\rho$ by
\begin{equation}\label{def:Jrho}
  \rho(J^\rho \cdot, \cdot) := \rho(\cdot, J\cdot)
\end{equation}
and a self-dual two-form $\om^\rho$ with respect to $g^\rho$ by
\begin{equation}\label{def:OmRho}
  \om^\rho := R^\rho \om.
\end{equation}
Then
$$
g^\rho = \om^\rho(\cdot, J^\rho\cdot)
$$
and
\begin{equation}\label{eq:starRhoJrhoIdentity}
  *^\rho ( \lambda \wedge \om^\rho) = \lambda \circ J^\rho
\end{equation}
for all one-forms $\lambda \in \Om^1(M)$.

Fix a symplectic form $\rho \in \sS_a$. Let $\sS_{a+\cH^{\rho}}$ be the
space of symplectic forms representing a cohomology class in the affine space
$a + \cH^{\rho} \subset H^2(M, \R)$, where $\cH^{\rho}$ are the harmonic self-dual
forms with respect to the Hodge star operator
$*^{\rho}$. Define the map $K: \sS_{a+\cH^{\rho}} \to \Om^+$ by
$$
K(\rho) := \frac{\rho + * \rho}{2u}, \qquad u = \frac{\rho \wedge \rho}{2 \dvol}.
$$
Denote the extension of this map to the Sobolev space $\sS^{k,p}_{a+\cH_\rho}$
also by $K$.
\begin{theorem}[{\bf The Map $K$}]\label{thm:K}
  Let $k - \frac{4}{p}> 0$.

  \smallskip\noindent{\bf (i)}
  For every $\rho \in \sS^{k,p}_{a + \cH_{\rho}}$
  there exists a $W^{k,p}$-neighborhood of $\rho$ such that $K$
  restricted to this neighborhood is a diffeomorphism of Banach spaces.

  \smallskip\noindent{\bf (ii)}
  Let $\cH \in H^2(M;\R)$ be a positive linear subspace and $a\in H^2(M; \R)$. The map
  $K: \sS_{a + \cH} \to \Om^+$ is injective.

  \smallskip\noindent{\bf (iii)}
  There exists polynomials $\fp_1, \fp_2$ with positive coefficients
  with the following significance. If $\rho \in \sS^{k,p}$ and $K(\rho) \in
  W^{k+1,p}(M,\Lambda^+)$ with $\frac{1}{u} \leq C < \infty$, then $\rho \in
  \sS^{k+1,p}$ and
  \begin{multline*}
    \norm{\rho}_{W^{k+1,p}}
    \leq
    \fp_1\left(C, \norm{\rho}_{L^\infty}  \right)
    \norm{\frac{\rho^+}{u}}_{W^{k+1,p}}\\
    + \fp_2\left(C,
    \norm{\rho}_{L^\infty}, \norm{\rho}_{W^{k-1,p}} ,
    \norm{\frac{\rho^+}{u}}_{W^{k,p}}\right)\norm{\rho}_{W^{k,p}}
  \end{multline*}
\end{theorem}
\begin{proof}
  See page \pageref{proof:K}.
\end{proof}

We will need the following three lemmas to prove Theorem~\ref{thm:K}.

\begin{lemma}[{\bf Negative chords}] \label{lem:negativeChords}
  Let $V$ be a real four-dimensional vector space equipped with the standart
  metric, $\theta \in \Lambda^+V$ with $\theta^2 = \dvol$ and $\rho_1,\rho_2 \in
  \Lambda^2V$ such that
  \begin{gather*}
    \rho_i \wedge \rho_i = \dvol, \qquad \rho_i^+ = \lambda_i
    \theta,\qquad \lambda_i \geq 1, \qquad i=1,2.
  \end{gather*}
  Then
  $$
  \left( \rho_1 - \rho_2 \right)^2 \leq 0,
  $$
  and
  $$
  \left( \rho_1 - \rho_2 \right)^2 = 0 \Leftrightarrow \rho_1 = \rho_2.
  $$
\end{lemma}
\begin{proof}
  Since $\rho_1^+ = \lambda_1 \theta$ and $\rho_2^+ = \lambda_2\theta$ for
  $\lambda_1, \lambda_2 \geq 1$ we find
  \begin{align*}
    \left( \rho_1 - \rho_2 \right)^2 &=\left( \rho_1^+ - \rho_2^+ + \rho_1^- -
    \rho_2^- \right)^2\\
    &=
    \left( \left( \lambda_1 - \lambda_2 \right)\theta + \rho_1^- -
    \rho_2^-\right)^2\\
    &=
    \left( \lambda_1 - \lambda_2 \right)^2\dvol + \left( \rho_1^- - \rho_2^-
    \right)^2\\
    &=
    \left( \lambda_1 - \lambda_2 \right)^2\dvol - \Abs{\rho_1^-}^2\dvol -
    \Abs{\rho_2^-}^2\dvol -2 \rho_1^-\wedge \rho_2^-.
  \end{align*}
  Thus,
  $$
  \frac{\left( \rho_1 - \rho_2 \right)^2}{\dvol} = \left( \lambda_1 -
  \lambda_2 \right)^2 - \Abs{\rho_1^-}^2 - \Abs{\rho_2^-}^2 +
  2\left<\rho_1^-,\rho_2^-\right>.
  $$
  Since
  $$
  \dvol = \rho_i\wedge \rho_i = \lambda_i^2 \theta\wedge \theta -
  \rho_i^- \wedge \rho_i^- = \left(\lambda_i^2 - \Abs{\rho_i^-}^2\right) \dvol
  $$
  we have $\Abs{\rho_i^-}^2 = \lambda_i^2 -1$ and hence
  \begin{align*}
    \frac{\left( \rho_1 - \rho_2 \right)^2}{\dvol}
    &=
    \lambda_1^2 + \lambda_2^2 -
    2\lambda_1\lambda_2 - \lambda_1^2 +1 - \lambda^2_2 + 1 +
    2\left<\rho_1^-,\rho_2^-\right>\\
    &=
    -2\lambda_1\lambda_2 + 2 + 2\left<\rho_1^-,\rho_2^-\right>.
  \end{align*}
  By the Cauchy-Schwarz inequality $\left<\rho_1^-,\rho_2^-\right> \leq
  \Abs{\rho_1^-}\Abs{\rho_2^-}$ and equality holds if and only if $\rho_1^-$ and
  $\rho_2^-$ are colinear. Thus,
  \begin{equation}\label{eq:cs}
    \begin{split}
      \frac{\left( \rho_1 - \rho_2 \right)^2}{\dvol}
      &\leq
      -2\lambda_1\lambda_2 + 2 + 2\Abs{\rho_1^-}\Abs{\rho_2^-}\\
      &=
      -2\lambda_1\lambda_2 + 2 + 2\sqrt{\lambda_1^2 -1} \sqrt{\lambda_2^2 -1}
    \end{split}
  \end{equation}
  and equality holds only if $\rho_1^-$ and $\rho_2^-$ are colinear. Suppose
  $\left(\rho_1 - \rho_2\right)^2 \geq 0$. Then
  $$
  \lambda_1\lambda_2 - 1 \leq \sqrt{\lambda_1^2 -1} \sqrt{\lambda_2^2 -1}.
  $$
  Squaring both sides of the inequality yields
  $$
  \lambda_1^2\lambda_2^2 - 2\lambda_1\lambda_2 + 1 \leq (\lambda_1^2 - 1)
  (\lambda_2^2 -1 ) = \lambda_1^2\lambda_2^2 -\lambda_1^2 - \lambda_2^2 + 1.
  $$
  Hence we find
  $$
  \lambda_1^2 + \lambda_2^2 - 2\lambda_1\lambda_2 \leq 0
  $$
  with equality if and only if $\rho_1^-$ and $\rho_2^-$ are colinear. But
  clearly
  $$
  \lambda_1^2 + \lambda_2^2 - 2\lambda_1\lambda_2 = \left( \lambda_1 - \lambda_2
  \right)^2 \geq 0
  $$
  and it follows that $\lambda_1 = \lambda_2$ and $\rho_1^-$ and $\rho_2^-$ are
  colinear! From $\rho_i^+ = \lambda_i \theta$ we get $\rho_1^+= \rho_2^+$.
  Further, since $\lambda_1 = \lambda_2$ and $\Abs{\rho_i^-}^2 =
  \lambda_i^2 - 1$ we get $\Abs{\rho_1^-} =
  \Abs{\rho_2^-}$ and therefore $\rho_1^- = \pm\rho_2^-$. However, if $\rho_2 =
  \rho_1^+ - \rho_1^-$ and $\Abs{\rho_1^-} > 0$, then $\left( \rho_1 - \rho_2
  \right)^2 \le 0$ in violation to our assumption. We conclude that $\rho_1 =
  \rho_2$.
\end{proof}

\begin{remark}
  The set
  \begin{equation}\label{def:constraint}
    \cP_{\dvol, \theta} := \left\{\left.\rho \in \Lambda^2V \right| \rho\wedge\rho = \dvol, \rho^+ =
      \lambda \theta\text{ for } \Abs{\lambda}\ge 1\right\}.
    \end{equation}
    for a fixed volume form $\dvol$ and a self-dual two-form $\theta$ was
    consider by Donaldson in $\cite{DON2}$ in the context of the following
    problem: How many symplectic structures $\rho$ exist, such that $\rho\wedge
    \rho = \dvol$ for a prescribed volume form $\dvol$, $\rho$ is compatible
    with a prescribed almost complex
    structure and $\rho$ lays in a given positive affine subspace of $H^2(M,
    \R)$? The answer is that it is unique.
    The set $\cP_{\dvol, \theta}$ is a three-dimensional submanifold of
    $\Lambda^2V$ with two components. The key property of this manifold
    established in Lemma~\ref{lem:negativeChords} is called `negative chords'.
  \end{remark}

  \begin{lemma}[{\bf The Linearization of $K$}]\label{lem:keyRegularity}
    Let $\rho_s$ be a path of nondegenerate 2-forms and $\rhohat =
    \left.\frac{d}{ds}\right|_{s=0} \rho_s$, $\rho = \rho_0$. Then
    \begin{equation}\label{eq:Khat}
      \left.\frac{d}{ds}\right|_{s=0}\left(\frac{\rho_s^+}{u_s}\right) =
      \frac{1}{u} R^\rho \rhohat^{+^\rho}, \qquad u_s = \frac{\rho_s\wedge
      \rho_s}{2\dvol}, \qquad R^{\rho} w = w - \frac{w \wedge \rho}
      {\dvol_{\rho}} \rho.
    \end{equation}
    In particular,
    $$
    \rhohat^{+^\rho} = u R^\rho
    \left.\frac{d}{ds}\right|_{s=0}\left(\frac{\rho_s^+}{u_s}\right).
    $$
  \end{lemma}
  \begin{proof}
    We compute
    \begin{equation*}
      \left.\frac{d}{ds}\right|_{s=0}\left(\frac{\rho_s^+}{u_s}\right) =
      \frac{\rhohat^+}{u} - \frac{\rhohat \wedge \rho}{\dvol_{\rho}}
      \frac{\rho^+}{u} = \frac{1}{u}\left(R^{\rho} \rhohat\right)^+ =
      \frac{1}{u}R^{\rho} \rhohat^{+_{\rho}}.
    \end{equation*}
    For the last equality we used that the linear
    map $R^{\rho}:\Lambda^2V \to \Lambda^2V$ is an involution on $\Lambda^2V$
    for a 4-dimensional real vector space $V$ and it maps $\Lambda^+ $ to
    $\Lambda^{+^{\rho}}$ and vice versa, where $\Lambda^{+^\rho}$ is the space
    of self-dual 2-forms for the metric $g^\rho$ (see \cite{KROMSAL} for a proof
    of these facts).
  \end{proof}

  \begin{lemma}[{\bf Lie Derivative}]\label{cor:keyRegularity}
    Let $X$ be a vector field on $M$ and $\rho$ a nondegenerate two-form. Then
    $$
    \left(\cL_X \rho\right)^{+^\rho} = u R^\rho \left(\cL_X \frac{\rho^+}{u}
    -\frac{\cL_X \dvol}{\dvol}\frac{\rho^+}{u} - \frac12\left(\cL_X
    *\right)\frac{\rho}{u}\right).
    $$
  \end{lemma}
  \begin{proof}
    Let $\psi_s$, $s\geq 0$ be the family of diffeomorphisms on $M$ generated by
    the vector field $X$. Then
    \begin{equation*}
      \begin{split}
        2 \psi_s^* \frac{\rho^+}{u}
        &=
        \psi_s^* \left(\left(\rho + *\rho\right)
        \frac{2\dvol}{\rho\wedge\rho}\right)\\
        &=
        \left( \psi_s^* \rho + \left(\psi_s^* * \right) \left(\psi_s^*
        \rho\right)\right) \frac{2 \psi_s^* \dvol}{\psi_s^* (\rho\wedge\rho)}\\
        &=
        \left(\psi_s^* \rho+ *\left(\psi_s^* \rho\right)
        \right)\frac{2\dvol}{\psi_s^* (\rho\wedge\rho)} \frac{2 \psi_s^*
        \dvol}{2\dvol}\\
        &\qquad
        +  \left( \left(\psi_s^* *\right) \left(\psi_s^* \rho\right)
        -*\left(\psi_s^* \rho\right) \right)\frac{2 \psi_s^* \dvol}{\psi_s^*
        (\rho\wedge\rho)}
      \end{split}
    \end{equation*}
    Thus
    \begin{equation*}
      \begin{split}
        \left.\frac{d}{ds}\right|_{s=0}\left(2 \psi_s^* \frac{\rho^+}{u}\right)
        &=
        \left.\frac{d}{ds}\right|_{s=0} \left(2\left(\psi_s^*\rho\right)^+
        \frac{2\dvol}{\psi_s^*\left(\rho\wedge\rho\right)}\right) +
        \frac{2\rho^+}{u} \left.\frac{d}{ds}\right|_{s=0}\frac{\psi_s^*
        \dvol}{\dvol}\\
        &\qquad
        +\left. \frac{d}{ds}\right|_{s=0}\left(\psi_s^* *\right)\frac{\rho}{u}.
      \end{split}
    \end{equation*}
    Using Lemma~\ref{lem:keyRegularity} we then compute
    \begin{equation*}
      \begin{split}
        2\left(\cL_X \rho\right)^{+^\rho}
        &=
        2\left(\left.\frac{d}{ds}\right|_{s=0} ({\psi_s}^*
        \rho)\right)^{+^\rho}\\
        &=
        u R^\rho \left.  \frac{d}{ds}\right|_{s=0}
        \left(2(\psi_s^*\rho)^+\frac{2\dvol}{\psi_s^*(\rho\wedge\rho)}\right)\\
        &=
        u R^\rho\left(2\cL_X\frac{\rho^+}{u} - \frac{2\rho^+}{u}
        \frac{\cL_X\dvol}{\dvol}  -
        \left(\cL_X*\right)\frac{\rho}{u}\right)
      \end{split}
    \end{equation*}
    This proves the lemma.
  \end{proof}

  \begin{proof}[Proof of Theorem~\ref{thm:K}]\label{proof:K}
    We prove (i). By Lemma~\ref{lem:keyRegularity} the linearization of $K$ is
    given by
    $$
    \Khat: T_{\rho}\sS^{k,p}_{a+ \cH_{\rho}} \to W^{k,p}(M, \Lambda^{+}) :\qquad
    \rhohat \mapsto \frac{1}{u}R^\rho \rhohat^{+^\rho}
    $$
    We claim that this is an isomorphism. Then (i) follows from the inverse
    function theorem for Banach spaces. By Hodge theory for the operator
    $d^{*^{\rho}}d + dd^{*^{\rho}}$ every $\rho \in \sS^{k,p}_{a +
    \cH^{\rho}}$ can be written as a sum
    $
    \rho = \rho + d\lambda + h
    $
    for a unique $ \lambda \in W^{k+1,p}(M, T^*M) $
    and a $h \in \cH^{\rho}$ such that $\int_M d\lambda \wedge *^{\rho} h =0$ and
    $d^{*^{\rho}}\lambda =0$. Hence
    $$
    \rhohat = d\lambda + h
    $$
    for such unique $\lambda$ and $h$.  Now suppose $\Khat \rhohat = 0$. It
    follows that
    $$
    d^{(+^{\rho})}\lambda + h = \frac12 (d\lambda + *^{\rho}d\lambda) + h = 0.
    $$
    Since $h$ is closed, $dd^{(+^{\rho})}\lambda = 0$, and thus
    $$
    0 = \int_M dd^{+^{\rho}}\lambda \wedge \lambda = - \int_M
    d^{(+^{\rho})}\lambda \wedge d\lambda = - \int_M
    \Abs{d^{(+^{\rho})}\lambda}\dvol.
    $$
    Since $0 = \int_M d\lambda \wedge d\lambda = \int_M
    \Abs{d^{(+^{\rho})}\lambda} - \Abs{d^{(-^{\rho})}\lambda}\dvol$ it
    follows that
    $$
    d\lambda = 0.
    $$
    Together with $d^{*^{\rho}}\lambda = 0$ this implies that $\lambda $ is
    harmonic and $h=0$. This shows that $\Khat$ is injective. Now let $\eta \in
    W^{k,p}(M,\Lambda^+)$. Since $W^{k,p}$ is closed under products and
    composition with smooth functions for $k - \frac{4}{p} > 0$, $u_0R^{\rho}\eta$ is
    in $ W^{k,p}(M,\Lambda^{+^{\rho}})$ and by Hodge theory there exists a
    unique one-form $\lambda \in W^{k+1,p}(M, T^*M)$ and a harmonic two-form $h$
    which is self-dual with respect to $*^{\rho}$ such that $u_0 R^{\rho}\eta =
    d\lambda + h$. Then $\Khat \left( d\lambda + h \right) = \Khat \left( u_0
    R^{\rho} \eta \right) = \eta$ and this shows that $\Khat$ is surjective. This
    proves (i).

    We prove (ii). Let $\cH\subset H^2(M; \R)$ be a positive linear subspace. Let
    $\rho_1,\rho_2$ be two elements of $\sS_{a + \cH}$ such that
    $$
    \eta:=K(\rho_1) = K(\rho_2).
    $$
    Define
    $$
    \theta := \frac{\eta}{\Abs{\eta}^2} \in \Om^+.
    $$
    Then
    \begin{gather*}
      \theta\wedge \theta = \frac{1}{\Abs{\eta}^2}\dvol, \qquad \rho_i^+ =
      \lambda_i \theta, \qquad \lambda_i := {\Abs{\eta}^2u_{i}},
      \qquad u_i = \frac{\rho_i \wedge \rho_i}{2\dvol}
    \end{gather*}
    and
    \begin{equation*}
      \begin{split}
        (\rho_i - \theta)^2 &= \rho_i\wedge\rho_i - 2\rho_i^+ \wedge\theta + \theta^2\\
        &=
        2 u_i \dvol - 2 \lambda_i\theta^2 +\theta^2\\
        &=  2 u_i \dvol - 2 \Abs{\eta}^2 u_i \frac{1}{\Abs{\eta}^2}\dvol +\theta^2\\
        &=\theta^2
      \end{split}
    \end{equation*}
    for $i=1,2$. It now follows
    from Lemma~\ref{lem:negativeChords} that
    $$
    \left( \rho_1 - \rho_2 \right)^2  = \left((\rho_1 - \theta) - (\rho_2 -
    \theta) \right)^2\leq 0
    $$
    point wise.  On the other hand, since $[\rho_1], [\rho_2] \in a + \cH$ and
    $\cH$ is a positive subspace of $H^2(M)$, we have
    $$
    \int_M (\rho_1 - \rho_2)^2 \geq 0
    $$
    and therefore $(\rho_1 - \rho_2)^2 =0$. We conclude with
    Lemma~\ref{lem:negativeChords} that $\rho_1 = \rho_2$. This proves (ii).

    We prove (iii).  Let $\rho \in  \sS^{k,p}$ and let $\cL \rho$ be a Lie
    derivative of $\rho$ in an arbitrary direction. By Cartan's formula the Lie
    derivative of $\rho$ is exact and by Lemma~\ref{cor:keyRegularity}
    \begin{equation*}
      \begin{split}
        (d + d^{*^\rho}) \cL \rho
        &=
        d^{*^\rho} \cL \rho\\
        &=
        -*^\rho d *^\rho \cL \rho\\
        &=
        -*^\rho d\left( (\cL \rho)^{+^\rho} - (\cL \rho)^{-^\rho} \right)\\
        &=
        -2 *^\rho d (\cL \rho)^{+^\rho}\\
        &=
        - *^\rho d \left(u R^\rho \left(\cL \frac{\rho^+}{u} - \frac{\cL
        \dvol}{\dvol}\frac{\rho^+}{u}- \frac12\frac{(\cL *)
        \rho}{u}\right)\right).
      \end{split}
    \end{equation*}
    The right hand side is a term of the form
    $$
    P_1(\frac{1}{u}, \rho) \p^2 \frac{\rho^+}{u} + P_2(\frac{1}{u}, \rho) \p
    \rho \p \frac{\rho^+}{u} + P_3(\frac{1}{u}, \rho) \p \rho
    $$
    for polynoms $P_1, P_2, P_3$ with smooth coefficient functions in the
    indicated variables. It follows from elliptic regularity theory and the
    product estimates for Sobolev spaces of Lemma~\ref{lem:sobolevProd} that
    $\cL \rho \in W^{k,p}$ and that there exists polynomials $\fp_1, \fp_2$
    independent of $\rho$ with the following significance
    \begin{multline*}
      \norm{\cL \rho}_{W^{k,p}}
      \leq
      \fp_1\left(C, \norm{\rho}_{L^\infty}  \right)
      \norm{\frac{\rho^+}{u}}_{W^{k+1,p}}\\
      + \fp_2\left(C,
      \norm{\rho}_{L^\infty}, \norm{\rho}_{W^{k-1,p}} ,
      \norm{\frac{\rho^+}{u}}_{W^{k,p}}\right)\norm{\rho}_{W^{k,p}}
    \end{multline*}
    for $C = \sup_{x \in M} \frac{1}{u(x)}$.  Since the Lie derivative $\cL
    \rho$ was arbitrary, the result follows.
  \end{proof}

  \section{The Evolution Equation for \texorpdfstring{$K(\rho)$}{K(p)}}\label{sec:reducedEvoEq}

  In view of Theorem~\ref{thm:K} the Donaldson flow has an equivalent
  description on the space of self-dual two-forms, given by the evolution of
  $K(\rho) = \frac{\rho^+}{u}$. This evolution equation exposes the parabolic
  nature of the Donaldson flow and it is the key for the regularity theorems we
  will prove in the later sections.

  To obtain a global formula we introduce the operator $S^\rho$,
  \begin{equation}\label{def:Srho}
    S^\rho: \Om^1 \to \Om^+, \qquad S^\rho \lambda := - R^\rho d^{+^\rho}
    \lambda + u \nabla_{X_\lambda}\frac{\rho^+}{u},
  \end{equation}
  where $\lambda \in \Om^1$, $R^\rho$ is defined by~\eqref{def:Rrho} and
  $\rho(X_\lambda,\cdot) := \lambda$.  We say $\om_1,\om_2,\om_3 \in \Om^+$ form
  a standart local frame of $\Om^+$ if and only if locally
  \begin{equation*}
    \om_i \wedge \om_j = \left\{
      \begin{array}{cc}
        2 \dvol & i = j\\
        0 & i \neq j
      \end{array}
      \right..
    \end{equation*}

    \begin{theorem}[{\bf The Evolution of $\frac{\rho^+}{u}$}]\label{thm:evolutionK}
      \smallskip\noindent{\bf (i)}
      Suppose $\rho$ is a smooth solution to the Donaldson flow. Then the evolution
      of the 2-form $\frac{\rho^+}{u}$ is given by the equation
      \begin{equation}\label{eq:evoK}
        \p_t \frac{\rho^+}{u}
        =
        \frac{1}{u}R^\rho d^{+^\rho} *^\rho d \theta^\rho,\qquad u =
        \frac{\rho\wedge\rho}{2\dvol}, \qquad \theta^\rho = \frac{2\rho^+}{u} -
        \Abs{\frac{\rho^+}{u}}\rho.
      \end{equation}
      Here $R^\rho$ is defined by equation~\eqref{def:Rrho}.

      \smallskip\noindent{\bf (ii)}
      Equation \eqref{eq:evoK} is the same as
      $$
      \p_t \frac{\rho^+}{u} = -\frac{2}{u} S^\rho\left(S^\rho\right)^{*^\rho}
      \frac{\rho^+}{u} + \nabla_{X_{(*^\rho d \theta^\rho)}}
      \frac{\rho^+}{u},
      $$
      where $S^\rho$ is defined by \eqref{def:Srho},
      $\left(S^\rho\right)^{*^\rho}$ is the adjoint of $S^\rho$ with respect to
      the inner products $\int_M \inner{\cdot}{\cdot}_{g^\rho}\dvol$ on $\Om^1$ and
      $\int_M\inner{\cdot}{\cdot}_g \dvol$ on $\Om^+$ and
      $$
      \rho\left(X_{(*^\rho d \theta^\rho)},\cdot\right) := *^\rho d \theta^\rho.
      $$

      \smallskip\noindent{\bf (iii)}
      Let ${\om_1, \om_2,\om_3}$ form a local standart frame of $\Om^+$. Then the
      evolution of the functions
      $
      K_i := \frac{\rho \wedge \om_i}{\dvol_\rho}
      $
      is given by
      \begin{equation}\label{eq:localEvoEq}
        \begin{split}
          \p_t \sum_i K_i \om_i
          &=
          -\sum_{i,j,k\, \text{cyclic}} \left(\frac{1}{u}d^{*^\rho} d K_i -
          2\{K_j,K_k\}_\rho - \rho\left( X_{K_i}, \sum_\ell J_\ell X_{K_\ell}\right)
          \right)\om_i\\
          &\qquad
          - \frac{1}{u}E^\rho_\om K + {E'}^\rho_\om K.
        \end{split}
      \end{equation}
      Here $X_H$ denotes the Hamiltonian vector field of the function $H$. The
      bracket $\{\cdot,\cdot\}_\rho$ denotes the Poisson bracket with respect to
      the symplectic structure $\rho$. $E^\rho_\om$ and ${E'}^\rho_\om$ are error
      terms depending on the frame $\om_1, \om_2, \om_3$ that vanish whenever
      $\nabla \om_i = 0$ and are given by
      \begin{equation}\label{def:E}
        \begin{split}
          E^\rho_w f
          &:=
          \sum_{i,j} \inner{(S^\rho)^{*^\rho} \om_i}{df_j \circ J_j^\rho}_{\rho} \om_i +
          2S^{\rho}\sum_j f_j \left(S^\rho \om_j\right)^{*^\rho}\\
          &\qquad +
          \sum_{i,j,k\, \text{cyclic}} *^\rho \left(df_j \wedge d\om_k - df_k \wedge
          d \om_j\right)\om_i\\
          {E'}^\rho_\om f
          &:= \sum_{i,j} K_j df_i (X_{(S^\rho)^{*^\rho} \om_j})\om_i+
          \sum_i f_i \nabla_{X_{*^\rho d \theta^\rho}} \om_i.
        \end{split}
      \end{equation}
      $$
      $$
      for $f=(f_1, f_2,f_3) \in \Om^0(M, \R^3)$.

      \smallskip\noindent{\bf (iv)}
      Assume the hyperK\"ahler case. Then the evolution of the functions
      $
      K_i = \frac{\rho \wedge \om_i}{\dvol_\rho}
      $
      is given by
      \begin{multline*}
        \p_t \sum_i K_i \om_i\\
        =
        -\sum_{i,j,k\, \text{cyclic}} \left(\frac{1}{u}d^{*^\rho} d K_i -
        2\{K_j,K_k\}_\rho - \rho\left( X_{K_i}, \sum_\ell J_\ell X_{K_\ell}\right)
        \right)\om_i.
      \end{multline*}
    \end{theorem}
    \begin{proof}
      See page~\pageref{proof:evolutionK}.
    \end{proof}

    We need the following lemma on the properties of $S^\rho$ and its adjoint
    $(S^\rho)^{*^\rho}$.
    \begin{lemma}[{\bf The Operators $S^\rho$, $(S^\rho)^{*^\rho}$}]\label{lem:Srho}
      Let $\rho \in \sS_a$.

      \smallskip\noindent{\bf (i)}
      The adjoint of $S^{\rho}: \Om^1(M) \to \Om^+(M)$ with respect to the inner
      product $\int_M\inner{\cdot}{\cdot}_{g^\rho}\dvol$ on $\Om^1$ and
      $\int_M\inner{\cdot}{\cdot}_g\dvol$ on $\Om^+$ is given by
      \begin{equation}\label{def:SrhoStar}
        (S^\rho)^{*^\rho} \xi
        :=
        -d^{*^\rho}  \left(R^\rho \xi\right) + *^\rho
        \left(g\left(\nabla \frac{\rho^+}{u}, \xi \right) \wedge
        \rho \right),
      \end{equation}
      where $g\left( \nabla \frac{\rho^+}{u}, \xi \right)$ is the 1-form given by
      $X \mapsto g\left( \nabla_X \frac{\rho^+}{u}, \xi \right)$ for a vector field
      $X$.

      \smallskip\noindent{\bf (ii)}
      $(S^\rho)^{*^\rho} \frac{2\rho^+}{u} = *^\rho d\theta^\rho$, where $\theta^\rho =
      \frac{2\rho^+}{u} + \Abs{\frac{\rho^+}{u}}^2\rho$.

      \smallskip\noindent{\bf (iii)}
      Let $\om \in \Om^+$. Then
      $$
      (S^\rho)^{*^\rho} \om = *^\rho \left( d\om - g\left( \nabla
      \om,\frac{\rho^+}{u} \right)\wedge \rho \right)
      $$
      In particular $(S^\rho)^{*^\rho} \om$ is independent of derivatives of
      $\rho$ and if $\nabla \om = 0$, then $(S^\rho)^{*^\rho} \om = 0$.

      \smallskip\noindent{\bf (iv)}
      The operator $(S^\rho)^{*^\rho}$ satisfies the following Leipniz rule,
      $$
      (S^\rho)^{*^\rho} (f \xi) = *^\rho(df \wedge R^\rho \xi) + f
      (S^\rho)^{*^\rho} \xi
      $$
      for all $\xi \in \Om^+$ and $f \in \Om^0(M, \R)$.

      \smallskip\noindent{\bf (v)}
      Let ${\om_1, \om_2,\om_3}$ be a standard frame for $\Om^+$ with respect to
      the background metric $g$ and $f_i \in \Om^0(M, \R)$, then
      $$
      (S^\rho)^{*^\rho} \left(\sum_i f_i \om_i\right) = \sum_i \left(d f_i \circ
      J_i^\rho + f_i (S^\rho)^{*^\rho} \om_i\right).
      $$

      \smallskip\noindent{\bf (vi)}
      In the hyperK\"ahler case,
      $$
      (S^\rho)^{*^\rho} \left(\sum_i f_i \om_i\right) = \sum_i d f_i \circ
      J_i^\rho.
      $$

      \smallskip\noindent{\bf (vii)}
      Let ${\om_1, \om_2, \om_3}$ be a standard frame for $\Om^+$ with respect
      to the background metric $g$ and $f_i \in \Om^0(M, \R)$, then
      \begin{multline*}
        2S^{\rho}(S^{\rho})^{*^\rho} \left( \sum_j f_j \om_j \right)\\
        =
        \sum_{i,j,k\, \text{cyclic}} \left(d^{*^\rho} d f_i -
        u\left(\{f_j,K^\rho_k\}_\rho - \{K^\rho_j,f_k\}_\rho\right)\right)  \om_i
        +  E^\rho_w
        f,
      \end{multline*}
      where
      $K_i^\rho := \frac{\om_i \wedge \rho}{\dvol_\rho}$ and $E^\rho_\om: \Om^0(M,
      \R^3) \to \Om^+ $ is the linear first order differential operator given
      by~\eqref{def:E}. It vanishes whenever $\nabla \om_i = 0$ for $i=1,2,3$.
    \end{lemma}
    \begin{proof}
      We prove (i). To compute the adjoint of $S^\rho$ let $\lambda$ be a 1-form
      and $\xi \in \Om^+$, then
      \begin{equation*}
        \begin{split}
          &\int_M \inner{-d^{*^\rho} \left(R^\rho \xi\right) + *^\rho\left( g\left(
          \nabla\frac{\rho^+}{u}, \xi \right)\wedge \rho  \right)}{\lambda}_{g^\rho} \dvol\\
          &\qquad =
          \int_M \inner{\xi}{-R^\rho d^{+^\rho} \lambda}_g\dvol + \int_M
          \inner{\rho\left(Y_{g\left(\nabla{\frac{\rho^+}{u}},\xi\right)}, \cdot\right)}{\rho(X_\lambda,
          \cdot)}_{g^\rho}\dvol,
        \end{split}
      \end{equation*}
      where
      $$
      g\left(Y_{g\left(\nabla{\frac{\rho^+}{u}},\xi \right)},\cdot\right) :=
      g\left(\nabla{\frac{\rho^+}{u}},\xi\right), \qquad \rho(X_\lambda, \cdot) :=
      \lambda.
      $$
      Here we used the identity
      $$
      *^\rho \iota(X)\rho = -\rho \wedge g(X, \cdot)
      $$
      in the last equation. By the definition of the Donaldson metric
      \begin{equation*}
        \begin{split}
          \int_M \inner{\rho\left(Y_{g\left(\nabla{\frac{\rho^+}{u}},\xi\right)},
        \cdot\right)}{\rho(X_\lambda, \cdot)}_{g^\rho}\dvol
        &=
        \int_M g\left(Y_{g\left(\nabla{\frac{\rho^+}{u}},\xi\right)}, X_\lambda
        \right)\dvol_\rho\\
        &=
        \int_M g\left(\nabla_{X_\lambda}{\frac{\rho^+}{u}},\xi\right)\dvol_\rho\\
      \end{split}
    \end{equation*}
    Thus we have proved that
    \begin{multline*}
      \int_M \inner{-d^{*^\rho} \left(R^\rho \xi\right) + *^\rho\left( g\left(
      \nabla\frac{\rho^+}{u}, \xi \right)\wedge \rho  \right)}{\lambda}_{g^\rho} \dvol\\
      =
      \int_M \inner{\xi}{-R^\rho d^{+^\rho} \lambda + u
      \nabla_{X_\lambda}{\frac{\rho^+}{u}}}_g\dvol
    \end{multline*}
    for all self-dual two-forms $\xi \in \Om^+(M)$ and one-forms $\lambda \in
    \Om^1(M)$. This proves (i).

    Part (ii) follows from the computation
    \begin{equation*}
      \begin{split}
        (S^\rho)^{*^\rho} \frac{2\rho^+}{u}
        &=
        {*^\rho}dR^{\rho}\frac{2\rho^+}{u} + *^\rho \left(2 g\left(
        \nabla\frac{\rho^+}{u},\frac{\rho^+}{u} \right)\wedge\rho \right)\\
        &=
        {*^\rho}d\left(R^{\rho}\frac{2\rho^+}{u} +
        \Abs{\frac{\rho^+}{u}}^2\rho\right)\\
        &=
        {*^\rho}d\left(\frac{2\rho^+}{u} - \frac{\rho \wedge
          \frac{2\rho^+}{u}}{\dvol_\rho}\rho +
          \Abs{\frac{\rho^+}{u}}^2\rho\right)\\
          &=
          {*^\rho}d\left(\frac{2\rho^+}{u} -
          \Abs{\frac{\rho^+}{u}}^2\rho\right)\\
          &=
          {*^\rho}d\theta^\rho,
        \end{split}
      \end{equation*}
      where we used that $*^\rho R^\rho \frac{2\rho^+}{u} = R^\rho
      \frac{2\rho^+}{u}$ in the first equation. This proves (ii).

      We prove (iii). Let $\om \in \Om^+$. Observe that
      \begin{equation*}
        \begin{split}
          d\left(R^\rho \om\right)
          &=
          d \left(\om - \frac{\om \wedge \rho}{\dvol_\rho} \rho\right)\\
          &=
          d\om  - \left(g\left( \nabla \om, \frac{\rho^+}{u} \right) + g\left(\om,
          \nabla \frac{\rho^+}{u}\right)\right) \wedge \rho.
        \end{split}
      \end{equation*}
      and hence
      \begin{equation*}
        \begin{split}
          (S^\rho)^{*^\rho} \om
          &=
          *^\rho d \left(R^\rho \om \right)+ *^\rho\left( g\left(
          \nabla\frac{\rho^+}{u},\om \right)\wedge \rho \right)\\
          &=
          *^\rho \left(d\om  -\left(g\left( \nabla \om, \frac{\rho^+}{u}
          \right)\wedge \rho \right)\right)
        \end{split}
      \end{equation*}
      Since $\nabla \om =0$ implies $d\om =0$ it follows from the last equation that
      $(S^\rho)^{*^\rho} \om = 0$. This proves (iii).

      We prove (iv). Let $\xi \in \Om^+$ and $f \in \Om^0(M, \R)$. Then
      \begin{equation*}
        \begin{split}
          (S^\rho)^{*^\rho}\left( f \xi \right)
          &=
          *^\rho d \left( R^\rho f\xi \right) + *^\rho\left( g\left( \nabla
          \frac{\rho^+}{u}, f\xi \right)\wedge \rho \right)\\
          &=
          *^\rho \left( df \wedge R^\rho \xi \right) + f *^\rho d\left( R^\rho
          \xi\right)
          + f *^\rho\left( g\left( \nabla \frac{\rho^+}{u}, \xi \right)\wedge \rho
          \right)\\
          &=
          *^\rho \left( df \wedge R^\rho \xi \right) + f (S^\rho)^{*^\rho} \xi.
        \end{split}
      \end{equation*}
      This proves (iv).

      We prove (v). It follows from (iv) that
      \begin{equation*}
        \begin{split}
          (S^\rho)^{*^\rho} (\sum_i f_i \om_i)
          &=  \sum_i\left(*^\rho  \left(df_i \wedge R^\rho
          \om_i\right) +  f_i (S^\rho)^{*^\rho} \om_i\right)\\
          &=
          \sum_i \left(df_i \circ J_i^\rho +
          f_i (S^\rho)^{*^\rho} \om_i\right),
        \end{split}
      \end{equation*}
      where the last equality follows from
      identity~\eqref{eq:starRhoJrhoIdentity}. This proves (v).

      (vi) follows directly from (v) since $(S^\rho)^{*^\rho} \om_i = 0$ by (iii)
      for the hyperK\"ahler structures $\om_1, \om_2, \om_3$.

      We prove (vii).
      It follows from (v) that in an standard frame
      $\om_1,\om_2,\om_3$ for $\Om^+$, $S^{\rho}$ is given by
      $$
      S^{\rho}\lambda = \frac{1}{2}\sum_i \left( -d^{*^\rho}\left(
      \lambda \circ J_i^\rho
      \right) + \inner{(S^\rho)^{*^\rho}\om_i}{\lambda}_{\rho}\right)\om_i.
      $$
      for a 1-form $\lambda$.  Again by (v)
      \begin{equation*}
        \begin{split}
          2S^\rho\left(S^{\rho}\right)^{*^\rho}  \sum_j f_j \om_j
          &=
          2S^{\rho}(\sum_j df_j \circ J_j^\rho  + f_j (S^\rho)^{*^\rho}
          \om_j)\\
          &=
          \sum_{i,j} -d^{*^\rho}(df_j \circ J_j^\rho \circ J_i^\rho)\om_i\\
          &\quad
          + \sum_i \inner{(S^\rho)^{*^\rho} \om_i}{\sum_j df_j \circ J_j^\rho}_{\rho} \om_i +
          2S^{\rho}\sum_j f_j \left(S^\rho \right)^{*^\rho}\om_j.
        \end{split}
      \end{equation*}
      Observe that
      $$
      d\left( R^\rho \om_i \right) = d\left( \om_i - K_i^\rho \rho \right) = d\om_i
      - dK^\rho_i \wedge \rho.
      $$
      Then by identity~\eqref{eq:starRhoJrhoIdentity} and the hyperK\"ahler
      relations for $J_i^\rho$
      \begin{align*}
        &d {*^\rho} (\sum_j df_j \circ J_j^\rho \circ J_1^\rho)
        =
        d{*^\rho}\left(-df_1 - df_2 \circ J^\rho_3 + df_3\circ J^\rho_2
        \right)\\
        &\quad=
        -d {*^\rho} df_1 + d \left(df_2\wedge R^\rho \om_3  - df_3\wedge R^\rho
        \om_2 \right)\\
        &\quad=
        -d{*^\rho}df_1 - df_2\wedge dK^\rho_3 \wedge \rho  + df_3\wedge
        dK^\rho_2\wedge\rho + df_2\wedge d\om_3 - df_3 \wedge d\om_2 \\
        &\quad=
        -d{*^\rho}df_1 - \iota\left(X_{f_2}\right) \rho \wedge \iota\left(
        X_{K^\rho_3}\right)\rho \wedge \rho  + \iota\left( X_{f_3} \right)\rho \wedge
        \iota\left( X_{K^\rho_2} \right)\rho \wedge\rho\\
        &\qquad \qquad
        + df_2\wedge d\om_3 - df_3 \wedge d\om_2\\
        &\quad=
        -d{*^\rho}df_1 - \{f_2,K^\rho_3\}_\rho \dvol_\rho - \{K^\rho_2,f_3\}_\rho
        \dvol_\rho+ df_2\wedge d\om_3 - df_3 \wedge d\om_2
      \end{align*}
      and hence
      \begin{equation*}
        \begin{split}
          \sum_{i,j} -d^{*^\rho}(df_j \circ J_j^\rho \circ J_i^\rho)\om_i
          &=
          \sum_{i,j,k\, \text{cyclic}} \left(d^{*^\rho} d f_i -
          u\left(\{f_j,K^\rho_k\}_\rho - \{K^\rho_j,f_k\}_\rho\right)\right)  \om_i\\
          &\qquad
          + \sum_{i,j,k\, \text{cyclic}} *^\rho \left(df_j \wedge d\om_k -
          df_k \wedge d \om_j\right)\om_i.
        \end{split}
      \end{equation*}
      This proves (vii).
    \end{proof}

    We end the section with a proof of Theorem~\ref{thm:evolutionK}.
    \begin{proof}[Proof of Theorem~\ref{thm:evolutionK}] \label{proof:evolutionK}
      We prove (i). Let $\rho$ be a smooth solution to the Donaldson flow. Recall
      that the Donaldson flow equation is $\p_t \rho = d *^\rho d \theta^\rho$.
      Then by equation~\eqref{eq:Khat}
      $$
      \p_t \frac{\rho^+}{u} = \frac{1}{u} (R^\rho \p_t \rho)^+ = \frac{1}{u} R^\rho
      (\p_t \rho)^{+^\rho}= \frac{1}{u} R^\rho d^{+^\rho}*^\rho d \theta^\rho.
      $$
      Here we use that $R^\rho$ maps $\Lambda^+$ to $\Lambda^{+^\rho}$ and vise
      versa in the second equation.  This proves (i).

      We prove (ii). By the definition of the operator $S^\rho$ given by
      \eqref{def:Srho} and Lemma~\ref{lem:Srho} (ii),
      \begin{equation*}
        \begin{split}
          -\frac{2}{u}S^\rho \left(S^\rho\right)^{*^\rho}\frac{\rho^+}{u} =
          -\frac{1}{u}S^\rho\left(*^\rho d\theta^\rho\right) =
          \frac{1}{u}R^\rho d^{+^\rho}*^\rho d\theta^\rho  - \nabla_{X_{(*^\rho
          d\theta^\rho})}\frac{\rho^+}{u}.
        \end{split}
      \end{equation*}
      Together with (i), this proves (ii).

      We prove (iii). First note that in a local standard frame $\om_1, \om_2,
      \om_3$ for $\Om^+$ we have $ \frac{2\rho^+}{u} = \sum_i K_i \om_i$
      and by Lemma~\ref{lem:Srho} (ii) and (v)
      $$
      *^\rho d \theta^\rho = (S^\rho)^{*^\rho} \frac{2\rho^+}{u} = \sum_i \left( d K_i \circ
      J_i^\rho + K_i (S^\rho)^{*^\rho} \om_i \right).
      $$
      Since
      $$
      \rho(J_i X_{K_i},\cdot) = \rho(X_{K_i}, J_i^\rho \cdot)
      $$
      the vector field $X_{(*^\rho d \theta^\rho)}$ that satisfies $\rho(X_{(*^\rho d
      \theta^\rho)}, \cdot)= *^\rho d \theta^\rho$ is given by
      $$
      \sum_j\left( J_j X_{K_j} + K_j X_{(S^\rho)^{*^\rho} \om_j}\right).
      $$
      Hence,
      \begin{equation*}
        \begin{split}
          \nabla_{X_{*^\rho d\theta^\rho}}\frac{2\rho^+}{u}
          &=
          \nabla_{X_{*^\rho d\theta^\rho}}\sum_i K_i \om_i\\
          &=
          \sum_i \left(dK_i(\sum_j J_j X_{K_j}) + \sum_j K_j dK_i
          (X_{(S^\rho)^{*^\rho}
        \om_j})\right)\om_i \\
        &\qquad + \sum_i K_i \nabla_{X_{*^\rho d \theta^\rho}} \om_i\\
        &=
        \sum_i \rho\left( X_{K_i}, \sum_j J_j X_{K_i} \right)\om_i +
        \sum_{i,j} K_j dK_i (X_{(S^\rho)^{*^\rho}
      \om_j})\om_i\\
      &\qquad
      + \sum_i K_i \nabla_{X_{*^\rho d \theta^\rho}} \om_i\\
      &=
      \sum_i \rho\left( X_{K_i}, \sum_j J_j X_{K_i} \right)\om_i +
      {E'}^\rho_\om K
    \end{split}
  \end{equation*}
  for $K = (K_1, K_2, K_3)$.
  By Lemma~\ref{lem:Srho} (vii)
  \begin{equation*}
    \begin{split}
      \frac{2}{u}S^\rho \left(S^\rho\right)^{*^\rho}\frac{2\rho^+}{u}
      &=
      \frac{2}{u}S^\rho\left(S^\rho\right)^{*^\rho}\left( \sum_j K_j \om_j\right)\\
      &=
      \sum_{i,j,k\, \text{cyclic}} \left(\frac{1}{u}d^{*^\rho} d K_i -
      2\{K_j,K_k\}_\rho\right) \om_i + \frac{1}{u}E^\rho_\om K.
    \end{split}
  \end{equation*}
  Thus, from
  $$
  \p_t \frac{2\rho^+}{u} = -\frac{2}{u} \left( S^\rho \right)^{*^\rho}S^\rho
  \frac{2\rho^+}{u} + \nabla_{X_*^\rho d\theta^\rho}\frac{2\rho^+}{u}
  $$
  we find that
  \begin{equation*}
    \begin{split}
      \p_t \sum_i K_i \om_i
      &=
      -\sum_{i,j,k\, \text{cyclic}} \left(\frac{1}{u}d^{*^\rho} d K_i -
      2\{K_j,K_k\}_\rho - \rho\left( X_{K_i}, \sum_\ell J_\ell
      X_{K_\ell}\right) \right)\om_i\\
      &\qquad
      - \frac{1}{u}E^\rho_\om K + {E'}^\rho_\om(K).
    \end{split}
  \end{equation*}
  Finally, (iv) follows from (v), since $\nabla \om_i = 0$ for the three
  hyperK\"ahler structures $\om_1, \om_2, \om_3$ and therefore the error
  terms $E^\rho_\om$ and ${E'}_\om$ vanish.
\end{proof}

\section{Regularity}\label{sec:regularity}
In this section we prove that a solution to the Donaldson flow that is element
of $L^2(I, W^{2,p}) \cap W^{1,2}(I, L^p)$ for $p>4$ is as smooth as it's
initial condition allows. In particular it is smooth if its initial conditions
are smooth. The proof combines two insights. First, the regularity of
$\frac{\rho^+}{u}$ determines the regularity of $\rho$. This is the content of
theorem~\ref{thm:K} (iii). Second, the evolution of $\frac{\rho^+}{u}$ is
given by a parabolic operator, where the right hand side of the equation is
essentially a product of two derivatives. This is the content of
Theorem~\ref{thm:evolutionK} (iii). This allows bootstrapping. The details are
given in the next two theorems. The first theorem illustrates the ideas in the
simpler case of a critical point.

\begin{theorem}[\bf Critical Point]\label{thm:criticalPoints}
  Let $p>4$ and let $\rho \in W^{1,p}(M,
  \Lambda^2)$ be a critical point of the Donaldson flow. Then $\rho$ is
  smooth.
\end{theorem}

The next lemma will be needed in the bootstrapping process.
\begin{lemma}\label{lem:regKtoRho}
Let $p'> p > 4$. If $\rho \in W^{1,p}$ is a symplectic form with $
\frac{\rho^+}{u} \in W^{1,p'} $ then $\rho \in W^{1,p'}$ as well.
\end{lemma}
\begin{proof}
Let $\cL \rho$ be a Lie derivative of $\rho$ in an arbitrary direction. Let
$\rho_0$ be a smooth nondegenerate form such that $\norm{\rho -
\rho_0}_{L^\infty} < \delta$ for a small $\delta >0$. Then there exists a
unique 1-form $\lambda \in W^{1,p}$ such that $d\lambda = \cL \rho$ and
$d^{*^{\rho_0}}\lambda = 0$. By elliptic regularity theory for the operator
$d^{+^{\rho_0}} + d^{*^{\rho_0}}$ there exists a constant $c>0$ such that
$$
\norm{\lambda}_{W^{1,{p'}}}\leq
c\left(\norm{(d\lambda)^{+^{\rho_0}}}_{L^{p'}} +
\norm{d\lambda}_{L^2}\right).
$$
Since
$$
\norm{(*^\rho - *^{\rho_0}) d\lambda}_{L^{p'}} \leq \norm{\rho -
\rho_0}_{L^\infty}\norm{d\lambda}_{L^{p'}}
$$
we have
$$
\norm{(d\lambda)^{+^{\rho_0}}}_{L^{p'}} \leq
\norm{(d\lambda)^{+^{\rho}}}_{L^{p'}} + \delta \norm{d\lambda}_{L^{p'}}
$$
Further, since
$$
0 = \int_M d\lambda \wedge d\lambda = \int_M \Abs{d
\lambda}^{+^{\rho}} - \Abs{d\lambda}^{-^{\rho}}.
$$
we have
$$
\norm{d\lambda}_{L^2} \leq 2 \norm{(d\lambda)^{+^\rho}}_{L^2}
$$
By Lemma~\ref{cor:keyRegularity},
$$
(\cL \rho)^{+^\rho} = u R^{\rho}\left(\cL \frac{\rho^+}{u} - \frac{\cL_X
\dvol}{\dvol} \frac{\rho^+}{u} - \frac12\frac{(\cL *) \rho}{u}\right).
$$
Therefore,
\begin{equation*}
\begin{split}
\norm{\lambda}_{W^{1,{p'}}} &\leq c
\left(\norm{(\cL \lambda)^{+^{\rho}}}_{L^{p'}} + \delta
\norm{d\lambda}_{L^{p'}} + 2 \norm{(d\lambda)^{+^{\rho}}}_{L^2}\right)\\
&\leq
\fp_1(C, \norm{\rho}_{L^\infty}) \norm{\cL \frac{\rho^+}{u}}_{L^{p'}} +
\fp_2(C, \norm{\rho}_{L^\infty}) + c \delta \norm{\lambda}_{W^{1,{p'}}}
\end{split}
\end{equation*}
for
$$
C:= \sup_{x \in M} \frac{1}{u(x,t)}, \qquad u := \frac{\rho\wedge
\rho}{2\dvol}.
$$
Since this is true for an arbitrary Lie derivative of $\rho$ and $\delta>0$,
it follows that
$$
\rho \in W^{1,p'}.
$$
This proves the lemma.
\end{proof}

We will need the following lemma on elliptic regularity of the operator
$d^{*^{\rho}}\frac{d}u : C^\infty(M) \to C^\infty(M)$ in the case that $\rho$
and thus the coefficients are not smooth.
\begin{lemma}[{\bf Elliptic Regularity}]\label{lem:ellipticRegularity}
  Let $p>4$, $q>1$ and $k\geq 0$. Let $\rho \in W^{k+1,p}(M, \Lambda^2)$ such
  that $\rho\wedge\rho >0$.  Assume there exists a constant $c_0>0$ such that
  for all $v,w\in C^\infty(M)$ and all $\epsilon > 0$ we can estimate
  \begin{equation}\label{eq:ellRegAssumption}
    \norm{\p v \p w}_{W^{k,q}} \leq c_0
    \norm{v}_{W^{k+1,p}}\left(\frac{1}{\epsilon}\norm{w}_{L^q} + \epsilon
    \norm{w}_{W^{k+2,q}}\right).
  \end{equation}
  Let $v \in W^{1,q}(M)$ and $f\in W^{k,q}(M)$ such that
  $$
  \int_M \frac{1}{u}d v \wedge *^\rho d \phi = \int_M f \phi \dvol
  $$
  for all $\phi \in C^1(M)$.  Then $v \in W^{k+2,q}(M)$ and there exists a
  constant $c=c(q,k,M, \norm{\rho}_{W^{1,p}})$ such that
  $$
  \norm{v}_{W^{k+2,q}} \leq c \left(\norm{f}_{W^{k,q}} +
  \norm{v}_{L^{q}}\right).
  $$
\end{lemma}
\begin{proof}
  We only proof the case $k=0$, the general case follows by induction over
  $k$. Choose coordinate charts for $M$ and a subordinate partition of unity
  of $M$. Let $\psi \in C^\infty_0(M)$ be a cutoff function. Then we have
  $$
  \int_M \frac{1}{u}d (\psi v) \wedge *^\rho d \phi = \int_M f' \phi \dvol
  $$
  for all $\phi \in C^1(M)$ and
  $$
  f' := \psi f - \left(d^{*^\rho} \frac{d}{u} \psi \right) v + *^\rho \left(dv
  \wedge *^\rho \frac{d}{u}\psi\right).
  $$
  Let $B \subset \R^4$ be a ball around zero. Let $\triangle$ be the Hodge
  laplacian on $\R^4$ with respect to the standart metric. We know from
  ellpitic regularity if $v \in W^{1,p}_0(B)$ is a weak solution to the
  equation
  $$
  \triangle v = f
  $$
  in the sense that for every $\phi \in C^1_0(B)$ we have
  $$
  \int_{B} dv \wedge * d \phi = \int_{B} f \phi \dvol
  $$
  for an $f \in L^{q}(B)$ and $q>1$, then $v \in W^{2,q}(B)$ and there exists
  a constant $c_1(q, B)$ such that
  $$
  \norm{v}_{W^{2,q}(B)} \leq c_1 \left(
  \norm{f}_{L^{q}(B)} + \norm{v}_{L^{q}(B)}\right).
  $$
  We choose a coordinate chart such that the image of the support of $\psi$ is
  contained in $B$. We can assume that the push forward of $\rho$ under this
  coordinate chart equals the standart symplectic structure at $0 \in B$. If
  not we can always achieve this by a change of coordinates. Let us denote the
  pushforward of $\rho$ under this coordinate by $\rho_\alpha$. Further we denote
  by $\triangle^{\rho_\alpha}$ the operator $d^{*^\rho}\frac{d}{u}$ expressed
  in this chart, by $v_\alpha$ the function $\psi v$ expressed in this chart
  and by $f'_\alpha$ the function $f'$ expressed in this chart. By
  estimate~\eqref{eq:ellRegAssumption} there exists a polynomial with positive
  coefficents $\fp$ such that
  \begin{equation}\label{eq:est1ellReg}
    \begin{aligned}
      &\norm{\left(\triangle^{\rho_\alpha} - \triangle\right) v_\alpha}_{L^{q}}\\
      &\quad\leq
       \fp(\norm{\rho_\alpha}_{L^\infty})\left(\norm{\rho_\alpha - \rho(0)}_{L^\infty}
       \norm{v_\alpha}_{W^{2,q}} + \norm{\p \rho_\alpha \p
       v_\alpha}_{L^{q}}\right)\\
      &\quad\leq
       \fp(\norm{\rho_\alpha}_{L^\infty})\bigg(\norm{\rho_\alpha - \rho(0)}_{L^\infty}
       \norm{v_\alpha}_{W^{2,q}}\\
       &\qquad\qquad +
      \norm{\rho_\alpha}_{W^{1,p}} \left(\frac{1}{\epsilon}\norm{ v_\alpha}_{L^q} +
      \epsilon\norm{v_\alpha}_{W^{2,q}}\right)\bigg)
    \end{aligned}
  \end{equation}
  Further, by interpolating $\norm{\p v_\alpha}_{L^q} \leq c
  \norm{v_\alpha}^{\frac12}_{L^q} \norm{v_\alpha}^{\frac12}_{W^{2,q}}$ with the
  Galgliardo-Nirenberg interpolation inequality~\ref{prop:gagliardo} we find
  \begin{equation}\label{eq:est2ellReg}
    \norm{f'_\alpha}_{L^q} \leq c_2 \norm{f}_{L^q} + c_3
    \left(\frac{1}{\epsilon}\norm{v_\alpha}_{L^q} + \epsilon
    \norm{v_\alpha}_{W^{2,q}}\right).
  \end{equation}
  If $v_\alpha$ is a smooth solution to
  $$
  \triangle^{\rho_\alpha}v_\alpha = f'_\alpha
  $$
  then it also solves
  $$
  \triangle v_\alpha = f'_\alpha + \left(\triangle -
  \triangle^{\rho_\alpha}\right) v_\alpha.
  $$
  By choosing $\epsilon$ and the ball $B$ small we see that there exits a
  constant constant $c_4 = c_4(q,\norm{\rho}_{W^{1,p}})$ such that
  \begin{equation}\label{eq:ellipticEstimate}
    \norm{v}_{W^{2,q}} \leq c_4 \left(\norm{f}_{L^q} + \norm{v}_{L^q} \right).
  \end{equation}
  Here we are using that $ W^{1,p}(M) \hookrightarrow C^0(M) $ for $p>4$ and
  therefore $\norm{\rho_\alpha - \rho_0}_{L^\infty(B)}$ is small for a small
  ball. If $v \in W^{1,q}(M)$ is a weak solution to $d^{*^\rho}\frac{d}{u} v =
  f$ in the sense that
  $$
  \int_M \frac{1}{u}d v \wedge *^\rho d \phi = \int_M f \phi \dvol
  $$
  for all $\phi \in C^1(M)$ then we approximate $f$ and $\rho$ by smooth
  functions $f_k$ respectively smooth nondegenerate two-forms $\rho_k$, such
  that
  $$
  \lim_{k\to \infty}\norm{f_k - f}_{L^q} = 0, \qquad \lim_{k\to
  \infty}\norm{\rho - \rho_k}_{W^{1,p}} =0.
  $$
  For each pair $(f_k, \rho_k)$ we find by standart $L^2$-theory for elliptic
  operators a smooth function $v_k$ that solves
  $$
  d^{*^{\rho_k}}\frac{d}{u_k}v_k = f_k.
  $$
  The constant $c_4$ in~\eqref{eq:ellipticEstimate} can be choosen uniform in
  $k$ for big $k$ and it follows that $\{v_k\}_{k\in \N}$ has a weakly
  convergent subsequence with limit $\bar{v} \in W^{2,q}$ that satisfies the
  estimate~\eqref{eq:ellipticEstimate} and $d^{*^\rho}\frac{d}{u} \bar{v} =
  f$. Hence
  $$
  \int_M \frac{1}{u} d (v - \bar{v}) \wedge *^\rho d\phi= 0
  $$
  for all $\phi \in C^1(M)$. By choosing a sequnce $\phi_k \in C^1(M)$ such
  that $\phi_k$ converges to $(v - \bar{v})$ in $W^{1,2}(M)$ we
  see that $\bar{v} = v$. This proves the lemma.
\end{proof}
\begin{proof}[Proof of Theorem~\ref{thm:criticalPoints}]
  If $\rho$ is a critical point of the Donaldson flow then it follows
  from Theorem~\ref{thm:evolutionK} (iii) that for a local standard frame
  $\om_1, \om_2,\om_3$ and
  $$
  K_i = \frac{\rho \wedge \om_i} {\dvol_\rho}
  $$
  we have the equation
  $$
  \frac{1}{u}d^{*^\rho} d K_i = -2\{K_j,K_k\}_\rho - \rho\left( X_{K_i},
  \sum_\ell J_\ell X_{K_\ell}\right) + \frac{\left(\frac{1}{u} E_\om^\rho -
  {E'_\om}^\rho\right) \wedge \om_i}{2\dvol}
  $$
  and cyclic permutations of $i,j,k$. The right hand side of this equation
  consists of products of derivatives of the $K_i$ functions times a
  polynomial in the $\rho$ and $\frac{1}{u}$ variables plus lower order terms
  in the $K_i$ functions times another polynomial of the same form. Thus,
  schematically we may write
  \begin{equation}\label{eq:proofCritPoint1}
    d^{*^\rho} d K = P_1(\frac{1}{u},\rho) \p K\cdot \p K +
    P_2(\frac{1}{u},\rho) \p K
  \end{equation}
  Since $1-\frac{4}{p}>0$, $\rho \in C^0$ and since we assume that $\rho $ is
  a symplectic structure we have $\sup_M \frac{1}{u} < \infty$. It follows
  that the $L^\infty$-norms of $P_1,P_2$ are bounded. Using H\"older's
  inequality we see that the right hand side is element of $L^{\frac{p}{2}}$.
  For two functions $v,w \in C^\infty(M)$ we have the estimate
  \begin{equation}\label{eq:ellRegAssumptionCritPoint}
    \norm{\p v \p w}_{L^{\frac{p}{2}}} \leq \norm{\p v}_{L^p} \norm{\p
    w}_{L^{p}} \leq \norm{v}_{W^{1,p}} \left( \frac{1}{\epsilon} \norm{v}_{L^p} +
    \epsilon \norm{w}_{W^{2,p}}\right)
  \end{equation}
  where we used H\"olders inequality in the first inequality and the
  Gagliardo-Nirenberg interpolation inequality~\ref{prop:gagliardo} in the
  second.  It then follows from Lemma~\ref{lem:ellipticRegularity} that $K$ is
  element of $W^{2,\frac{p}{2}}$. By Rellich's embedding theorem
  $W^{1,\frac{p}{2}} \hookrightarrow L^{p'}$ where
  $$
  p' = \frac{4p}{8 - p}
  $$
  For $4 < p < 8$ we have
  $$
  p' > p.
  $$
  Thus $K \in W^{1,p'}$ and by Lemma~\ref{lem:regKtoRho} $\rho \in W^{1,p'}$.
  If we repeat this argument with $p$ replaced by $p'$, we find that $K \in
  W^{1,\frac{p''}{2}}$, $p'' > p'$ and
  $$
  p'' - p' > p' - p.
  $$
  Hence, eventually we find that $K\in W^{2, q}$ and $\rho \in W^{1,q}$ for
  all $q \geq p$. Now we can use Theorem~\ref{thm:K} to see that $\rho \in
  W^{2,q}$ as well. This implies that the right hand side of
  equation~\eqref{eq:proofCritPoint1} is element of $W^{1,q}$ and elliptic
  regularity gives us that $K \in W^{3,q}$.  Now an
  obvious iteration of these arguments using elliptic regularity and
  Theorem~\ref{thm:K}, to deduce the regularity of $\rho$ from the regularity
  of $K$, proves the theorem.
\end{proof}

Let $T>0$ and
$$
I:= [0,T) \subset \R.
  $$
  For $p>4$, and integers $k\geq 2$, $0 \leq r \leq \lfloor\frac{k}{2}\rfloor$ define the
  Sobolev space $W^{r,k,p}$ of functions on $I\times M$ such that the weak
  derivatives $\p^s_t \p^\ell$ exist and are bounded in the $L^p - L^2$-norm for
  $2s + \ell \leq k$, $s \leq r$,
  \begin{multline*}
    W^{r,k,p}(M_I):= \\
    L^2\left(I, W^{k,p}(M)\right) \cap W^{1,2}\left(I,
    W^{k-2,p}(M)\right) \cap \cdots \cap W^{r,2}\left(I, W^{k-2r, p}(M)\right).
  \end{multline*}
  This definition extends in an obvious way to functions from $I$ to the space
  of sections of a vector bundle over $M$. In the case at hand the relavant
  vector bundle is the bundle of two-forms over $M$ and the corresponding space
  will be denoted by
  $$
  W^{r,k,p}(I, \Lambda^2).
  $$
  If the vector bundle in question is clear, we will simply write $W^{r,k,p}$.
  The norm on this space is given by
  \begin{equation}\label{def:norm}
    \norm{u}_{W^{r, k,p}} := \sum_{\substack{2s + \ell \leq
    k \\ s \leq r}} \left(\int_M \left(\norm{\p_t^s
    u}_{W^{\ell,p}}\right)^2 dt\right)^\frac12.
  \end{equation}
  \begin{remark}[{\bf Besov Spaces}]
    \begin{enumerate}[1)]
      \item
        The restriction map $\rho \mapsto \rho(t=0,\cdot)$ extends to a bounded
        linear operator
        $$
        W^{r,k,p}(M_I, \Lambda^2) \to B^{k-1,p}_2(M, \Lambda^2),
        $$
        where $B^{k-1,p}_2(M, \Lambda^2)$ denotes the Besov space (see
        \cite{Grigoryan}) with exponents $p$ and $q = 2$.
      \item
        This restriction map is surjective and it has a bounded right inverse.
      \item
        The identity operator on real smooth functions on $I$ to sections in the
        vector bundle $\Lambda^2$ over $M$ with compact support extends to a
        bounded linear operator
        $$
        W^{r,k,p}(M_I, \Lambda^2) \to C^0(I, B^{k-1, p}_2(M, \Lambda^2))
        $$
      \item
        It holds
        $$
        W^{k,2}(M, \Lambda^2) = B^{k,2}_2(M, \Lambda^2).
        $$
    \end{enumerate}
  \end{remark}
  \begin{theorem}[{\bf Flow Lines}]\label{thm:regularity}
    Let $\rho \in W^{1,2}(I, L^p) \cap L^2(I, W^{2,p})$ be a solution to the
    Donaldson flow for $p > 4$ with initial condition $\rho(t=0,\cdot) =
    \rho_0$. For every integer $k\geq 1$ and all $4 < p' < p$ the following are
    equivalent:
    \begin{enumerate}[1)]
      \item
        $\rho_0 \in B^{k,p'}_2(M, \Lambda^2)$.
      \item
        $\rho \in W^{\frac{k+1}{2}, k+1, p'}(M, \Lambda^2)$.
    \end{enumerate}
  \end{theorem}

  The proof of this theorem uses parabolic regularity theory. In particular we
  use the `maximal regularity' property of parabolic operators in divergence
  form. We refer to \cite{Lamberton} for these results. The maximal regularity
  property is usually formulated for operators with time independent smooth
  coefficients for operators in divergence form, the Hodge laplacian being the
  archetypal example.  The next lemma assures that the operator
  $$
  d^{*^\rho}\frac{d}{u} : C^\infty(M) \to C^\infty(M)
  $$
  has the maximal
  regularity property as well, even though it's coefficients depend on time and
  are non-smooth in our use case.
  \begin{lemma}[{\bf Maximal Regularity}] \label{lem:maximalRegularity}
    Let $p> 4$, $q > 1$. Let
    $$
    \rho \in W^{1,2}(\R, L^p(M, \Lambda^2)) \cap L^2(\R,
    W^{2,p}(M, \Lambda^2))
    $$ be a path of nondegenerate forms. Assume that there exists a constant
    $c_0 > 0$ such that for all $v,w \in C^\infty(M)$ and all $\epsilon >0$ we
    can estimate
    \begin{equation}\label{eq:maxRegAssumption}
      \norm{\p v \p w}_{L^q} \leq c_0 \norm{v}_{W^{1,p}} \left(\epsilon
      \norm{w}_{W^{2,q}} + \frac{1}{\epsilon} \norm{w}_{L^q}\right).
    \end{equation}
    Then for all $v\in
    C_0^\infty(\R, C_0^\infty(M))$ there exists a
    constant $c(q,\norm{\rho}_{L^\infty(\R, W^{1,p})})>0$ such that
    $$
    \norm{\p_t v}_{L^2(\R, L^q)}\leq c \norm{\p_t v + d^{*^\rho}
    \frac{d}{u}v}_{L^2(\R, L^{q})} + \norm{v}_{L^2(\R, L^q)}.
    $$
  \end{lemma}
  \begin{proof}
    Choose coordinate charts for $\R \times M$ and a subordinate partition of
    unity. Let $\psi \in C^\infty(\R \times M)$ be a cutoff function. Let
    $I\times B\in \R \times \R^4$ be a ball around $(0,0)$. Choose a
    coordinate chart such that the support of $\psi$ is mapped into $I\times
    B$ and such that the pushforward of $\rho$ under this coordinate chart at
    $(0,0)$ equals the standart symplectic structure on $\R^4$. We can always
    achieve this by a change of coordinates. Let us denote the pushforward of
    $\rho$ under this coordinate chart by $\rho_\alpha$ and the operator
    $d^{*^\rho}\frac{d}{u}$ expresed in this chart by $\triangle^{\rho_\alpha}$. Further
    we denote by $v_\alpha$ the function $\psi v$ expressed in this chart.
    From maximal regularity for the standart Laplace operator
    $\triangle$ on $\R \times \R^4$ there exists a constant $c_1 = c_1(q)>0$
    such that \begin{align*}
      \norm{\p_t v_\alpha}_{L^2(\R, L^q)}
      &\leq c_1
      \norm{\p_t v_\alpha + \triangle v_\alpha}_{L^2(\R, L^q)}\\
      &\leq c_1
      \norm{\p_t v_\alpha + \triangle^{\rho_\alpha} v_\alpha}_{L^2(\R, L^q)} +
      \norm{(\triangle - \triangle^{\rho_\alpha})v_\alpha}_{L^2(\R, L^q)}
    \end{align*}
    By choosing the partition of unity such that $\psi$ has small enough
    support and using the estimate~\eqref{eq:est1ellReg} we find that there
    exists a constant
    $$
    c_2 = c_2 (q, \norm{\rho}_{L^\infty(\R, W^{1,p})})
    $$
    such that
    $$
    \norm{\p_t v_\alpha}_{L^2(\R, L^q)} \leq c_2 \norm{\p_t v_\alpha +
      \triangle^{\rho_\alpha}v_\alpha}_{L^2(\R, L^q)} + \norm{v_\alpha}_{L^2(\R,
      L^q)}.
    $$
    From this the global estimate
    $$
    \norm{\p_t v}_{L^2(\R, L^q)} \leq c_3 \norm{\p_t v +
      d^{*^\rho}\frac{d}{u}v}_{L^2(\R, L^2)} + \norm{v}_{L^2(\R, L^q)}
    $$
    folows by choosing a partition of unity such that the previous estimates
    holds for all chart domains and by estimating additional first order terms
    appearing from the multiplication of $v$ with the cutoff functions with
    the Gagliardo-Nirenberg interpolation inequality. This proves the lemma.
  \end{proof}

\begin{proof}[Proof of Theorem \ref{thm:regularity}]
Let $\rho(t=0,\cdot) = \rho_0 \in B^{2,p}_2$ and let
$$
\rhotilde_0 \in W^{1,2}(I, W^{1,p}) \cap L^2(I, W^{3,p})
$$
be an extension of $\rho_0$ with $\rhotilde_0(t=0,\cdot) = \rho_0$.  From
Theorem~\ref{thm:evolutionK} (iii) it follows that the evolution of the
functions $K_i = \frac{\rho\wedge\om_i}{\dvol_\rho}$ in a local standard
frame $\om_1,\om_2,\om_3$ is given by
\begin{equation*}
\p_t K_i  + \frac{1}{u} d^{*^\rho}d K_i\\
=  P_1(\frac{1}{u}, \rho) \p K \cdot\p K +
P_2(\frac{1}{u} ,\rho) \p K,
\end{equation*}
where $K = K_k$ for an arbitrary $k\in \{1,2,3\}$, $\p$ is an arbitrary
space derivative and $P_{1,2}(\frac{1}{u}, \rho)$ are polynomials in the
variables $\frac{1}{u}$ and $\rho$ with coefficients that are smooth
functions on the manifold. Let
$$
\Ktilde_{i,0} := \frac{\om_i \wedge \rhotilde_0}{\dvol_{\rhotilde_0}},
\qquad \Khat_i := K_i - \Ktilde_{i,0}
$$
The evolution of the functions $\Khat_{i}$ is then given by the equations
\begin{equation}\label{eq:reg1}
\begin{split}
\p_t \Khat_i  + d^{*^\rho}\frac{d}{u} \Khat_i
&=
P_3(\frac{1}{u}, \rho) \p \rho \cdot\p K + P_2(\frac{1}{u} ,\rho) \p K\\
&\qquad - (\p_t + \frac{1}{u} d^{*^\rho}d)\Ktilde_{0,i}\\
\Khat_i(t=0,\cdot) &= 0,
\end{split}
\end{equation}
for a polynom $P_3$ with smooth coefficient functions. By assumption $\rho$
is element of $W^{1,2}(I, L^p) \cap L^2(I, W^{1,p})$ and therefore $\rho, K
\in C^0(I, W^{1,p})$. It follows with H\"older's inequality that
$$
\p \rho \cdot \p K \in L^{\frac{q}{2}}(I, L^\frac{p}{2}) \qquad
\forall q \geq 0.
$$
Since $W^{1,p} \subset C^0$ for $p > 4$, we have
$$
\norm{P_{2,3}}_{L^\infty(I, L^\infty)} < \infty.
$$
The last term on the right handside of~\eqref{eq:reg1} is in $L^2(I,
W^{1,p}) \subseteq L^2(I, L^\infty)$, hence the right hand side
of~\eqref{eq:reg1} is element of $L^{2}(I, L^{\frac{p}{2}})$. As in the
critical point case the estimate~\eqref{eq:ellRegAssumptionCritPoint} is valid
for any two functions $v,w\in C^\infty(M)$.  Then by the ellipitic regularity
Lemma \ref{lem:ellipticRegularity} and the maximal regularity
Lemma \ref{lem:maximalRegularity} we have,
$$
\Khat \in W^{1, 2}\left( I, L^{\frac{p}{2}} \right) \cap
L^{2}\left( I, W^{2, \frac{p}{2}} \right).
$$
And by Rellich's theorem,
$$
\Khat \in C^0(I, W^{1, p'}), \qquad p' = \frac{4p}{8-p}.
$$
For $4<p <8$ we have
$$
p' > p.
$$
Since $\Ktilde_0 \in C^0(I, W^{2,p}) \subseteq C^0(I, C^1)$ it follows that
$K \in C^0(I, W^{1,p'})$. Then with Lemma~\ref{lem:regKtoRho} we find
$$
\rho \in C^0(I, W^{1, p'}).
$$
If we now repeat these arguments with $p$ replaced by $p'$ we find that
$\Khat \in C^0(I, W^{1,p''})$ with $p'' > p'$ and
$$
p'' - p' > p' - p.
$$
Hence, we eventually find that
$$
\Khat, K \in C^0(I, W^{1,q})
$$
for all $q \geq 1$.  In particular, the right hand side of \eqref{eq:reg1}
is element of $L^2(I, W^{1,p'''})$ for a $4 < p''' < p$ and
$$
K\in W^{1,2}(I, W^{1,p'''}) \cap L^2(I, W^{3,p'''}).
$$
We claim that the following
implications hold for $k\geq3$, $p>4$,
\begin{gather*}
K \in W^{\lfloor \frac{k}{2} \rfloor, k ,p}, \rho \in W^{\lfloor
\frac{k-1}{2} \rfloor, k-1 ,p} \Rightarrow \rho \in W^{\lfloor
\frac{k}{2} \rfloor, k ,p},\\
\rho, K \in W^{\lfloor \frac{k}{2} \rfloor, k ,p} \Rightarrow
P_1(\frac{1}{u}, \rho) \p K \cdot\p K + P_2(\frac{1}{u} ,\rho) \p K \in
W^{\lfloor \frac{k-1}{2} \rfloor, k-1,
p}\\
P_1(\frac{1}{u}, \rho) \p K \cdot\p K + P_2(\frac{1}{u} ,\rho) \p K \in
W^{\lfloor \frac{k-1}{2} \rfloor, k-1, p},\ \rho \in W^{\lfloor \frac{k}{2}
\rfloor,k,p},\ \rho_0 \in B^{k,p}_2\\
\Rightarrow  K \in W^{\lfloor \frac{k+1}{2} \rfloor, k+1 ,p}.
\end{gather*}
The statement of the theorem then follows by induction over $k\geq 3$ and
$p=p'''$. We prove the first implication. It follows from
Theorem~\ref{thm:K} that $\rho \in L^2(I, W^{k,p})$ under the stated
assertions. Then from the Donaldson flow equation $$ \norm{\p_t
\rho}_{W^{\lfloor \frac{k}{2}\rfloor - 1, k-2,p}} = \norm{d*^\rho
d\theta^\rho}_{W^{{\lfloor \frac{k}{2}\rfloor - 1, k-2,p}}}.  $$ This
expression is bounded because of the product estimates of
Lemma~\ref{lem:sobolevProd} in the case $k=3 $ and
Corollary~\ref{cor:parabolicProduct} for higher $k$, since the expressions
$*^\rho, \theta^\rho$ are just given by polynomials in the variables $\rho,
\frac{1}{u}$. The second implication follows also by these product estimates
in Sobolev spaces. To see the third implication, note that there exists an
extension of $\rho_0$ to an element $\rhotilde_0$ such that $\Ktilde_0 \in
W^{\lfloor \frac{k+1}{2} \rfloor, k+1,p}$. It follows from
equation~\eqref{eq:reg1} and maximal regularity for the operator
$\frac{1}{u}d^{*^\rho}d$ that $\Khat \in W^{\lfloor \frac{k+1}{2} \rfloor,
k+1,p}$ and hence so is $K$. This proves the theorem.
\end{proof}

The following is an immidiate corollary to Theorem~\ref{thm:regularity}.
\begin{corollary}[Instant Regularity]
Let $\rho \in L^2(I, W^{1,p})\cap W^{1,2}(I, L^p)$ be a solution to the
Donaldson flow. For all $4 < p' < p$ the following holds true. If
$\rho(t=0,\cdot) \in B_2^{k,p'}$ then the map $t \mapsto \rho(t,\cdot)$ is a
continous map from (0,T) to $B^{k+1,p'}_2$. In particular, $\rho(t,\cdot)$
is smooth for all $t \in (0,T)$.
\end{corollary}
\begin{remark}
Theorem~\ref{thm:regularity} can be refined using the more sophisticated
machinery of maximal $L^p-L^q$-regularity and the theory of Besov spaces
\cite{Grigoryan}. The correct space for the initial condition for a solution
in the space
$$
L^q(I, W^{k, p}) \cap W^{1,q}(I, W^{k-1,p}) \cap \ldots \cap W^{k,q}(I,
L^p)
$$
is the Besov space $B^{s,p}_q$ for
$$
s := k - \frac{2}{q}.
$$
For $sp > 4$ a solution is continuous in space and time and this bound is
sharp. The argument of Theorem~\ref{thm:regularity} can be used to show that
a solution in this space but initial conditions in $B^{s+1,p}_q$ is in
$$
L^q(I, W^{k+1, p}) \cap W^{1,q}(I, W^{k,p}) \cap \ldots \cap W^{k,q}(I,
W^{1,p}).
$$
\end{remark}


\section{Local Existence}\label{sec:ste}
In this section we show that we can find solutions to the Donaldson flow that
exist for a short time. The method of proof is essentially to apply Banach's
fixed point theorem to the linearized flow equation and show that the
difference between the linear and the non-linear operator has quadratic
growth. We will find a solution in the set $\sS^{1,2,p}_a$ for a $p>4$ with
initial condition $\rho(t=0,\cdot)$ in the Besov space $B^{1,p}_2$ (see
\cite{Grigoryan} for a definition). The set $\sS^{1,2,p}_a$ is the set of
paths of sections $\rho \in W^{1,2,p}(M_T, \Lambda^2)$, such that
$\rho(t,\cdot)$ is a symplectic form and represents the cohomolgy class $a\in
H^2$ for all $t \in [0,T)$, where $W^{1,2,p}(M_T,\Lambda^2)$ is the Banach
space
$$
W^{1,2,p}(M_I, \Lambda^2) :=
W^{1,2}(I, L^p (M, \Lambda^2)) \cap L^2(I, W^{2,p}(M, \Lambda^2)), \quad I
:= [0,T).
$$
\begin{theorem}[{\bf Short Time Existence}]\label{thm:localExistence}
Let $p>4$. For all $\rho_0 \in B^{1,p}_2$ there exists a $T>0$ such that
there exists a unique solution $\rho \in \sS_{a}^{1,2,p}$ on the interval
$[0,T)$ to the Donaldson flow with $\rho(0,\cdot) = \rho_0$.
\end{theorem}
\begin{proof}
See page \pageref{proof:localExistence}
\end{proof}
Consider the operator $d\frac{d^{*^{\rho_0}}}{u_0}: d\Om^1 \to d\Om^1$. It is
self-adjoint with respect to the inner product
$$
\inner{\rhohat_1}{\rhohat_2} = \int_M \inner{\rhohat_1}{
\rhohat_2}_{g^{\rho_0}}\dvol
$$
on $d\Om^1$ and the
inner product defined by
$$
\inner{\xi_1}{\xi_2} = \int_M \inner{\xi_1}{ \xi_2}_{g^{\rho_0}}\dvol_{\rho_0}
$$
on $\Om^1$.
It generates a strongly continuous semigroup on $d\left(W^{1,p}(M, \Lambda^1)
\right)$ with domain $d\left(W^{3,p}(\Lambda^1)\right)$ and it has the `maximal
regularity' property. TODO: Argument. In particular, there exists a unique
solution
$$
\rhotilde_0 \in L^2(I, W^{2,p}(M, \Lambda^2)) \cap W^{1,2}(I, L^p(M,
\Lambda^2))
$$
to the
homogeneous Cauchy problem
\begin{equation}\label{eq:homogenousCauchy}
\p_t \rhotilde_0 + d \frac{d^{*^{\rho_0}}}{u_0} \rhotilde_0 = 0, \qquad
\rhotilde(t = 0,
\cdot) = \rho_0.
\end{equation}
Further there exists a unique solution
$$
\rhohat\in L^2(I, W^{2,p}(M,
\Lambda^2)) \cap W^{1,2}(I,L^p(M, \Lambda^2))
$$
to the Cauchy problem
\begin{equation}\label{eq:inhomognousCauchy}
\p_t \rhohat + d \frac{d^{*^{\rho_0}}}{u_0}  \rhohat = f, \qquad \rhohat(t
= 0, \cdot) = 0
\end{equation}
for $f\in d(L^2(I, W^{1,p}))$.  Let
$$
B_r \subset d \left\{\lambda \in  W^{1,3,p}(M, \Lambda^1) |\,
\lambda(t=0,\cdot) = 0\right\}
$$
be a ball of radius $r>0$ centered at zero. Given $\rhohat \in B_r$, let
$\rhohat'$ be the unique solution to the equation
\begin{equation}\label{eq:steIteration}
\begin{gathered}
\p_t \rhohat' + d\frac{d^{*^{\rho_0}}}{u_0}\rhohat' = f(\rhohat),
\qquad \rhohat(t=0,\cdot) = 0\\
f(\rhohat):= \left(d*^{\rho}d\theta^\rho -
d*^{\rhotilde_0}d\theta^{\rhotilde_0} + L_{\rhotilde_0} \rhohat\right) +
d^{*^{\rhotilde_0}} d
\theta^{\rhotilde_0}\\
\qquad + \left(-A^{\rhotilde_0} + d\frac{d^{*^{\rho_0}}}{u_0} -
d\frac{d^{*^{\rhotilde_0}}}{\tilde{u}_0} \right) \rhohat\\ \rho :=
\rhotilde_0
+ \rhohat, \qquad u_0 = \frac{\rho_0 \wedge \rho_0}{2\dvol}, \qquad
\tilde{u}_0 := \frac{\rhotilde_0 \wedge \rhotilde_0}{2\dvol}, \qquad
\theta^\rho = \frac{2\rho^+}{u} - \Abs{\frac{\rho^+}{u}}^2 \rho.
\end{gathered}
\end{equation}
Here, $L_{\rhotilde_0}$ is the linearized gradient operator and $A^{\rho_0}$
is the lower order part of this operator. Both operators are given in the
following~Lemma \ref{lem:linearization} . This defines us a map $\sR$,
$$
\sR: d W^{1,3,p}(M, \Lambda^1) \to d W^{1,3,p}(M, \Lambda^1):\qquad \rhohat
\mapsto \rhohat'.
$$
If $\rhohat$ is a fixed point of $\sR$, then $\rho = \rho_0 + \rhohat$ is a
solution to the Donaldson flow. We will show that $\sR(B_r) \subseteq B_r$ for
a small $r>0$ and that $\sR$ is a contraction on $B_r$.
Theorem~\ref{thm:localExistence} will follow from these two claims by the
Banach fixed point theorem.

In preparation for the proof of Theorem~\ref{thm:localExistence} we need to
following four lemmas.  The first is the computation of the linearized
gradient operator. Recall that the gradient operator is given by
$$
\sS_a \to T\sS_a: \qquad \rho \to d*^\rho d \theta^\rho.
$$
\begin{lemma}[{\bf The linearized Operator}]\label{lem:linearization}
Let $\rho \in \sS_a$ and $\rho_s : (-1,1) \to \sS_a$ be a smooth path such
that $\rho_0 = \rho$ and $\left. \frac{d}{ds}\right|_{s=0} \rho_s =
\rhohat$. Then
\begin{equation}\label{eq:Lrho}
\left.\frac{d}{ds}\right|_{s=0} -d{*^{\rho_s}}d \theta^{\rho_s} =: L_\rho
\rhohat
=
d\frac{d^{*^\rho}}{u}\rhohat + A^\rho \rhohat
\end{equation}
where
\begin{equation}\label{eq:Arho}
A^\rho \rhohat :=  d*^\rho\left(
\left(\frac{du}{u^2}\right)\wedge \left( \rhohat + *^\rho \rhohat
\right)\right) +
d*^\rho \left(d\left( \Abs{\frac{\rho^+}{u}}^2 \right) \wedge
\rhohat\right)
-
d\hat{*}^{\rho,\rhohat} d\theta^\rho
\end{equation}
and $\hat{*}^{\rho,\rhohat} : \Lambda^3 \to \Lambda^{1}$ is the
linearization of the the Hodge star operator $*^{\rho_s}$ at $\rho$.
\end{lemma}
\begin{proof}
We proved in~\cite{KROMSAL} that
$$
\left. \frac{d}{ds}\right|_{s=0} \theta^{\rho_s} = \frac{\rhohat + *^\rho
\rhohat}{u} - \Abs{\frac{\rho^+}{u}}^2 \rhohat.
$$
Since $d\rhohat = 0$,
\begin{equation*}
\begin{split}
L_\rho \rhohat
&=
-d*^{\rho}d \left(\left.\frac{d}{ds}\right|_{s=0}\theta^{\rho_s}\right) -
d\hat{*}^{\rho,\rhohat} d\theta^\rho\\
&=
-d*^\rho \frac{1}{u} d *^\rho \rhohat + d*^\rho\left(
\left(\frac{du}{u^2}\right)\wedge \left( \rhohat + *^\rho \rhohat
\right)\right) +
d*^\rho \left(d\left( \Abs{\frac{\rho^+}{u}}^2 \right) \wedge
\rhohat\right)\\
&\qquad -
d\hat{*}^{\rho,\rhohat} d\theta^\rho\\
&=
d\frac{d^{*^\rho}}{u}\rhohat + A^\rho \rhohat.
\end{split}
\end{equation*}
This proves the lemma.
\end{proof}

The content of the next lemma are the quadratic estimates necessary to proof
existence of a solution to the Donaldson flow by the Banach fixed point theorem.
We establish these estimates in terms of polynomials with positive real
coefficients in several variables. We denote such a polynom by
$$
\mathfrak{p} (x_1,\ldots,x_\ell) = \sum_{\Abs{\alpha} \leq m} a_\alpha x^\alpha
$$
where the sum runs over all multi-indices $\alpha = (\alpha_1,\ldots,
\alpha_\ell) \in \N_0^\ell$ with $\Abs{\alpha} = \sum_{i=1}^\ell \alpha_i \leq
m$ and $a_\alpha \geq 0$ for all $\alpha$.
\begin{lemma}[{\bf Quadratic Estimates}]\label{lem:quadraticEstimate}
Let $p > 4$.

\smallskip\noindent{\bf (i)}
There exist two polynomials $\mathfrak{p_1}, \mathfrak{p_2}$ with positive
coefficients with the following significance.  For all $\rho \in \sS_a$
there exists a $\delta > 0$ such that for all $\rhohat \in d\Om^1$ with
$\norm{\rhohat}_{L^\infty} < \delta$ and $\rho':= \rho + \rhohat$
\begin{equation*}
\begin{split}
\norm{d*^{\rho'} d \theta^{\rho'} - d{*^\rho}d\theta^\rho - L_\rho
\rhohat}_{L^p}
&\leq
\mathfrak{p_1} \left(C, \norm{\rho}_{W^{1,p}}, \norm{\rhohat}_{W^{1,p}}
\right)\norm{\rhohat}_{W^{1,p}} \norm{\rhohat}_{W^{2,p}}\\
&\quad +
\mathfrak{p_2}(C, \norm{\rho}_{W^{1,p}},
\norm{\rhohat}_{W^{1,p}})\norm{\rho}_{W^{2,p}} \norm{\rhohat}_{W^{1,p}}^2.
\end{split}
\end{equation*}
where
$$
C := \sup_{\substack {x\in M \\ 0 \leq s \leq 1}}\frac{1}{u_{\rho(x) +
s\rhohat(x)}}, \qquad u_{\rho} := \frac{\rho\wedge \rho}{2\dvol}.
$$

\smallskip\noindent{\bf (ii)}
There exist two polynomials $\mathfrak{p}_3$, $\mathfrak{p}_4$ with positive
coefficients with the following significance.  For all $\rho \in \sS_a$
there exists a $\delta > 0$ such that if $\rho_1 := \rho + \rhohat_1$,
$\rho_2 :=
\rho + \rhohat_2$ for $\rhohat_1, \rhohat_2 \in d\Om^1$ with
$\norm{\rhohat_1}_{L^\infty}, \norm{\rhohat_2}_{L^\infty} < \delta$, then
\begin{equation*}
\begin{split}
&\norm{d*^{\rho_1}d\theta^{\rho_1} - d*^{\rho_2}d\theta^{\rho_2} + L_\rho \left(
\rhohat_1 - \rhohat_2 \right)}_{L^p}\\
&\leq
\mathfrak{p_3} \left(C, \norm{\rho}_{W^{1,p}}, \norm{\rhohat_1}_{W^{1,p}},
\norm{\rhohat_2}_{W^{1,p}}
\right) \left( \norm{\rhohat_1}_{W^{1,p}} + \norm{\rhohat_2}_{W^{1,p}}
\right)\norm{\rhohat_1 - \rhohat_2}_{W^{2,p}}\\
&\qquad +
\mathfrak{p_4}(C, \norm{\rho}_{W^{1,p}},
\norm{\rhohat_1}_{W^{1,p}},\norm{\rhohat_2}_{W^{1,p}})\left(\norm{\rho}_{W^{2,p}}
+ \norm{\rhohat_1}_{W^{2,p}} + \norm{\rhohat_2}_{W^{2,p}}\right)\\
&\qquad\qquad
\cdot \left( \norm{\rhohat_1}_{W^{1,p}} + \norm{\rhohat_2}_{W^{1,p}}
\right)\norm{\rhohat_2 - \rhohat_2}_{W^{1,p}}.
\end{split}
\end{equation*}
where
$$
C := \sup_{\substack {x\in M \\ 0 \leq s,s' \leq 1}}\frac{1}{u_{\rho(x) +
s(s'\rhohat_1(x) + (1 - s') \rhohat_2(x))}}.
$$
\end{lemma}
\begin{proof}
Let
$$
\rho_s = \rho + s \rhohat.
$$
Then,
\begin{equation}\label{eq:proofQE1}
\begin{split}
d*^{\rho'} d\theta^{\rho'} - d*^\rho d\theta^\rho + L_\rho \rhohat
&=
\int_0^1 \frac{d}{ds} \left(d*^{\rho_s}d\theta^{\rho_s}\right) ds - \left.
\frac{d}{ds}\right|_{s=0} d*^{\rho_s} d \theta^{\rho_s}\\
&= \int_0^1\int_0^{s}
\left(\frac{d^2}{d^2s'}d*^{\rho_{s'}}d\theta^{\rho_{s'}}\right) ds' ds.
\end{split}
\end{equation}
The integral in this equation is to be understood point wise in each fibre
of the vector bundle $\Lambda^2 T^*M$. By Minkowski's integral inequality,
it is enough to show that
$\norm{\frac{d^2}{d^2s}d*^{\rho_{s}}d\theta^{\rho_{s}}}_{L^p}$ is
bounded in the specified way. We compute,
\begin{equation}\label{eq:square}
\begin{split}
\frac{d^2}{d^2s} d*^{\rho_s}d\theta^{\rho_s}
&=
\frac{d}{ds} \left( d  \starhat^{\rho_s, \rhohat}d
\theta^{\rho_s} + d  *^{\rho_s} d\left(
\thetahat(\rho_s)\rhohat\right)\right)\\
&=
d\widehat{\starhat}^{\rho_s, \rhohat, \rhohat}d
\theta^{\rho_s} + 2 d  \starhat^{\rho_s,\rhohat} d\left(
\thetahat(\rho_s)\rhohat\right) + d *^{\rho_s} d \left(
\widehat{\thetahat}(\rho_s) \left(\rhohat ,\rhohat\right) \right),
\end{split}
\end{equation}
where
\begin{gather*}
\starhat^{\rho_s, \rhohat} := \frac{d}{ds} *^{\rho_s},\qquad
\widehat{\starhat}^{\rho_s, \rhohat, \rhohat} := \frac{d^2}{d^2s}
*^{\rho_s},\\
\qquad
\thetahat(\rho_s)\rhohat := \frac{d}{ds} \theta^{\rho_s}, \qquad
\widehat{\thetahat}(\rho_s) \left(\rhohat ,\rhohat\right) :=
\frac{d^2}{d^2s} \theta^{\rho_s}.
\end{gather*}
In what follows we suppress the notation of constants appearing in
inequalities. We estimate,
\begin{equation}\label{eq:squareEstimates}
\begin{split}
&\norm{\frac{d^2}{d^2s} d*^{\rho_s}d\theta^{\rho_s}}_{L^p}
\leq
\norm{\widehat{\starhat}^{\rho_s, \rhohat, \rhohat}d
\theta^{\rho_s}}_{W^{1,p}} + \norm{\starhat^{\rho_s,\rhohat} d\left(
\thetahat(\rho_s)\rhohat\right)}_{W^{1,p}}\\
&\qquad + \norm{*^{\rho_s} d \left(
\widehat{\thetahat}(\rho_s) \left(\rhohat
,\rhohat\right)\right)}_{W^{1,p}}\\
&\leq
\norm{\widehat{\starhat}^{\rho_s, \rhohat, \rhohat}}_{W^{1,p}} \norm{
\theta^{\rho_s}}_{W^{2,p}} +
\norm{\starhat^{\rho_s,\rhohat}}_{W^{1,p}} \norm{\left(
\thetahat(\rho_s)\rhohat\right)}_{W^{2,p}}\\
&\qquad + \norm{*^{\rho_s}}_{W^{1,p}}\norm{\left(
\widehat{\thetahat}(\rho_s) \left(\rhohat
,\rhohat\right)\right)}_{W^{2,p}},
\end{split}
\end{equation}
For ease of notation we don't number the different polynomials appearing in
the next estimates. We estimate the first of these summands by
\begin{equation*}
\begin{split}
&\norm{\widehat{\starhat}^{\rho_s, \rhohat, \rhohat}}_{W^{1,p}} \norm{
\theta^{\rho_s}}_{W^{2,p}}
\leq
\mathfrak{p} \left(C, \norm{\rho}_{W^{1,p}}, \norm{\rhohat}_{W^{1,p}}
\right)\norm{\rhohat}_{W^{1,p}}^2 \norm{
\theta^{\rho_s}}_{W^{2,p}}\\
&\qquad \leq
\mathfrak{p} \left(C, \norm{\rho}_{W^{1,p}}, \norm{\rhohat}_{W^{1,p}}
\right)\norm{\rhohat}_{W^{1,p}}^2\\
&\qquad \qquad \cdot \mathfrak{p}(C,
\norm{\rho}_{W^{1,p}}, \norm{\rhohat}_{W^{1,p}})(1 +
\norm{\rho_s}_{W^{2,p}}(1 +
\norm{\rho_s}_{L^\infty}))\\
&\qquad \leq
\mathfrak{p} \left(C, \norm{\rho}_{W^{1,p}}, \norm{\rhohat}_{W^{1,p}}
\right)\norm{\rhohat}_{W^{1,p}} \norm{\rhohat}_{W^{2,p}} \\
&\qquad \qquad +
\mathfrak{p}(C, \norm{\rho}_{W^{1,p}},
\norm{\rhohat}_{W^{1,p}})\norm{\rho}_{W^{2,p}}
\norm{\rhohat}_{W^{1,p}}^2.
\end{split}
\end{equation*}
Here the first inequality uses the product estimates of
Lemma~\ref{lem:sobolevProd} (i) and the fact that $\widehat{\starhat}^{\rho_s,
\rhohat, \rhohat}$ is quadratic in $\rhohat$. Further, by the definition of
$*^{\rho_s}$, $\widehat{\starhat}^{\rho_s}$ is a quadratic form given by a
polynomial in the variables $\frac{1}{u_s}$ and $\rho_s$, where $u_s =
\frac{\rho_s \wedge \rho_s}{2\dvol}$. Therefore we can estimate it's
$W^{1,p}$-norm by a polynomial $\mathfrak{p}(C, \norm{\rho_s}_{W^{1,p}})$
using Lemma~\ref{lem:sobolevProd} (ii). The second inequality follows again
from Lemma~\ref{lem:sobolevProd} (ii) now using that $\theta^{\rho_s}$ is a
polynomial in the variables $\frac{1}{u_s}$ and $\rho_s$. The third inequality
follows by grouping the appearing summands accordingly and using
$\norm{\rhohat}_{L^\infty} \leq \norm{\rhohat}_{W^{1,p}}$. With the same
arguments we estimate
\begin{equation*}
\begin{split}
&\norm{\starhat^{\rho_s,\rhohat}}_{W^{1,p}} \norm{\left(
\thetahat(\rho_s)\rhohat\right)}_{W^{2,p}}
\leq
\mathfrak{p}(C, \norm{\rho}_{W^{1,p}}, \norm{\rhohat}_{W^{1,p}})
\norm{\rhohat}_{W^{1,p}}\norm{\left(
\thetahat(\rho_s)\rhohat\right)}_{W^{2,p}}\\
&\qquad\leq
\mathfrak{p}(C, \norm{\rho}_{W^{1,p}}, \norm{\rhohat}_{W^{1,p}})
\norm{\rhohat}_{W^{1,p}}\\
&\qquad\qquad \cdot \left(\norm{\thetahat(\rho_s)}_{L^\infty}\norm{\rhohat}_{W^{2,p}} +
\norm{\thetahat(\rho_s)}_{W^{2,p}}\norm{\rhohat}_{L^\infty}\right)\\
&\qquad\leq
\mathfrak{p}(C, \norm{\rho}_{W^{1,p}}, \norm{\rhohat}_{W^{1,p}})
\norm{\rhohat}_{W^{1,p}} \norm{\rhohat}_{W^{2,p}}\\
&\qquad\qquad +\mathfrak{p}(C, \norm{\rho}_{W^{1,p}},
\norm{\rhohat}_{W^{1,p}}) \norm{\rhohat}_{W^{1,p}}\norm{\rhohat}_{L^\infty}(1 + \norm{\rho_s}_{W^{2,p}}(1 +
\norm{\rho_s}_{L^\infty}))\\
&\qquad \leq
\mathfrak{p} \left(C, \norm{\rho}_{W^{1,p}}, \norm{\rhohat}_{W^{1,p}}
\right)\norm{\rhohat}_{W^{1,p}} \norm{\rhohat}_{W^{2,p}}\\
&\qquad \qquad +
\mathfrak{p}(C, \norm{\rho}_{W^{1,p}}, \norm{\rhohat}_{W^{1,p}})\norm{\rho}_{W^{2,p}} \norm{\rhohat}_{W^{1,p}}^2
\end{split}
\end{equation*}
and
\begin{equation*}
\begin{split}
&\norm{*^{\rho_s}}_{W^{1,p}}\norm{\left( \widehat{\thetahat}(\rho_s)
\left(\rhohat ,\rhohat\right)\right)}_{W^{2,p}} \leq \mathfrak{p}(C,
\norm{\rho}_{W^{1,p}}, \norm{\rhohat}_{W^{1,p}}) \\
&\qquad \cdot \left(
\norm{\widehat{\thetahat}(\rho_s)}_{W^{2,p}}\norm{\rhohat}_{L^\infty}^2 +
\norm{\widehat{\thetahat}(\rho_s)}_{L^\infty}\norm{\rhohat}_{W^{2,p}}
\norm{\rhohat}_{L^\infty} \right)\\
&\qquad \leq
\mathfrak{p}(C, \norm{\rho}_{W^{1,p}}, \norm{\rhohat}_{W^{1,p}})
\norm{\rhohat}_{W^{1,p}} \norm{\rhohat}_{W^{2,p}}\\
&\qquad\qquad
+ \mathfrak{p}(C,
\norm{\rho}_{W^{1,p}}, \norm{\rhohat}_{W^{1,p}})\norm{\rhohat}^2_{L^\infty}(1 + \norm{\rho_s}_{W^{2,p}}(1 +
\norm{\rho_s}_{L^\infty}))\\
&\qquad \leq
\mathfrak{p} \left(C, \norm{\rho}_{W^{1,p}}, \norm{\rhohat}_{W^{1,p}}
\right)\norm{\rhohat}_{W^{1,p}} \norm{\rhohat}_{W^{2,p}}\\
&\qquad \qquad +
\mathfrak{p}(C, \norm{\rho}_{W^{1,p}}, \norm{\rhohat}_{W^{1,p}})\norm{\rho}_{W^{2,p}} \norm{\rhohat}_{W^{1,p}}^2.
\end{split}
\end{equation*}
This proves part (i).

We prove statement (ii). By~\eqref{eq:proofQE1},
\begin{equation*}
\begin{split}
&d*^{\rho_1}d\theta^{\rho_1} - d*^{\rho_2} d\theta^{\rho_2} - L_\rho \left(
\rhohat_1 - \rhohat_2 \right)\\
&\qquad =
d*^{\rho_1}d\theta^{\rho_1} - d*^\rho d \theta^\rho - L_\rho \rhohat_1 -
\left(d*^{\rho_2} d\theta^{\rho_2} - d*^\rho d \theta^\rho - L_\rho
\rhohat_2\right)\\
&\qquad =
\int_0^1\int_0^{s}\left(
\frac{d^2}{d^2s'}d*^{\rho_{1,s'}}d\theta^{\rho_{1,s'}}
-
\frac{d^2}{d^2s'}d*^{\rho_{2,s'}}d\theta^{\rho_{2,s'}}\right) ds' ds,
\end{split}
\end{equation*}
where
$$
\rho_{i,s'} := \rho + s' \rhohat_i
$$
for $i=1,2$. It follows from equation~\eqref{eq:square} that we need to
estimate the three terms
\begin{gather}\label{eq:quadTerm1}
\norm{\widehat{\starhat}^{\rho_{1,s},\rhohat_1, \rhohat_1}
d\theta^{\rho_{1,s}} - \widehat{\starhat}^{\rho_{2,s},\rhohat_2, \rhohat_2}
d\theta^{\rho_{2,s}}}_{W^{1,p}},\\
\label{eq:quadTerm2}
\norm{\starhat^{\rho_{1,s}, \rhohat_1} d \left(
\thetahat(\rho_{2,s})\rhohat_2 \right) - \starhat^{\rho_{2,s}, \rhohat_2} d \left(
\thetahat(\rho_{2,s})\rhohat_2 \right)}_{W^{1,p}},\\
\label{eq:quadTerm3}
\norm{*^{\rho_{1,s}}d \left( \widehat{\thetahat}(\rho_{1,s})\left(
\rhohat_1,\rhohat_2
\right) \right) - *^{\rho_{2,s}}d \left( \widehat{\thetahat}(\rho_{2,s})\left(
\rhohat_2,\rhohat_2
\right) \right)}_{W^{1,p}}
\end{gather}
in the specified way. We estimate the first term by
\begin{equation*}
\begin{split}
&\norm{\widehat{\starhat}^{\rho_{1,s},\rhohat_1, \rhohat_1}
d\theta^{\rho_{1,s}} - \widehat{\starhat}^{\rho_{2,s},\rhohat_2, \rhohat_2}
d\theta^{\rho_{2,s}}}_{W^{1,p}}\\
&\quad\leq
\norm{\widehat{\starhat}^{\rho_{1,s},\rhohat_1, \rhohat_1}
d\theta^{\rho_{1,s}} - \widehat{\starhat}^{\rho_{2,s},\rhohat_2, \rhohat_2}
d\theta^{\rho_{1,s}}}_{W^{1,p}}\\
&\qquad\qquad +
\norm{\widehat{\starhat}^{\rho_{2,s},\rhohat_2, \rhohat_2}
d\theta^{\rho_{1,s}} - \widehat{\starhat}^{\rho_{2,s},\rhohat_2,
\rhohat_2} d\theta^{\rho_{2,s}}}_{W^{1,p}}\\
&\quad\leq
\norm{\widehat{\starhat}^{\rho_{1,s},\rhohat_1, \rhohat_1} -
\widehat{\starhat}^{\rho_{2,s},\rhohat_2, \rhohat_2}}_{W^{1,p}}\norm{
\theta^{\rho_{1,s}}}_{W^{2,p}}
+
\norm{\widehat{\starhat}^{\rho_{2,s},\rhohat_2,
\rhohat_2}}_{W^{1,p}}
\norm{\theta^{\rho_{1,s}} - \theta^{\rho_{2,s}}}_{W^{2,p}}\\
\end{split}
\end{equation*}
Then,
\begin{equation*}
\begin{split}
&\norm{\widehat{\starhat}^{\rho_{1,s},\rhohat_1, \rhohat_1} -
\widehat{\starhat}^{\rho_{2,s},\rhohat_2, \rhohat_2}}_{W^{1,p}}\\
&\quad\leq
\norm{\widehat{\starhat}^{\rho_{1,s},\rhohat_1, \rhohat_1} -
\widehat{\starhat}^{\rho_{2,s},\rhohat_1, \rhohat_1} +
\widehat{\starhat}^{\rho_{2,s},\rhohat_1, \rhohat_1}
-\widehat{\starhat}^{\rho_{2,s},\rhohat_2, \rhohat_1} +
\widehat{\starhat}^{\rho_{2,s},\rhohat_2, \rhohat_1}-
\widehat{\starhat}^{\rho_{2,s},\rhohat_2, \rhohat_2}}_{W^{1,p}}\\
&\quad\leq
\norm{\widehat{\starhat}^{\rho_{1,s}} -
\widehat{\starhat}^{\rho_{2,s}}}_{W^{1,p}}\norm{\rhohat_1}_{W^{1,p}}^2\\
&\quad\quad +
\mathfrak{p}\left(C, \norm{\rho}_{W^{1,p}}, \norm{\rhohat_1}_{W^{1,p}},
\norm{\rhohat_2}_{W^{1,p}} \right)\left( \norm{\rhohat_1}_{W^{1,p}} +
\norm{\rhohat_2}_{W^{1,p}}\right)\norm{\rhohat_1 -
\rhohat_2}_{W^{1,p}}\\
&\quad \leq
\mathfrak{p}\left(C, \norm{\rho}_{W^{1,p}}, \norm{\rhohat_1}_{W^{1,p}},
\norm{\rhohat_2}_{W^{1,p}} \right)\left( \norm{\rhohat_1}_{W^{1,p}} +
\norm{\rhohat_2}_{W^{1,p}}\right)\norm{\rhohat_1 -
\rhohat_2}_{W^{1,p}},
\end{split}
\end{equation*}
using
\begin{equation*}
\begin{split}
\norm{\widehat{\starhat}^{\rho_{1,s}} -
\widehat{\starhat}^{\rho_{2,s}}}_{W^{1,p}}
&\leq
\norm{\int_0^1
\frac{d}{ds'}\widehat{\starhat}^{\left(s'\rho_{1,s} + (1-s')
\rho_{2,s}\right)}\, ds'}_{W^{1,p}}\\
&\leq
\mathfrak{p}\left( C, \norm{\rho}_{W^{1,p}}, \norm{\rhohat_1}_{W^{1,p}},
\norm{\rhohat_2}_{W^{1,p}} \right) \norm{\rhohat_1 - \rhohat_2}_{W^{1,p}}.
\end{split}
\end{equation*}
Further,
\begin{equation*}
\begin{split}
\norm{\theta^{\rho_{1,s}}}_{W^{2,p}}
&\leq \mathfrak{p}\left( C,
\norm{\rho}_{L^\infty}, \norm{\rhohat_1}_{L^\infty}
\right)\left(1 + \norm{\rho_{1,s}}_{W^{2,p}}\left(1 +
\norm{\rho_{1,s}}_{L^\infty}\right)\right)\\
&\leq
\fp \left(C, \norm{\rho}_{L^\infty} ,
\norm{\rhohat_1}_{L^\infty}\right)\\
&\qquad
+ \fp \left(, \norm{\rho}_{L^\infty} ,
\norm{\rhohat_1}_{L^\infty}\right) \left( \norm{\rho}_{W^{2,p}} +
\norm{\rhohat_1}_{W^{2,p}}\right)
\end{split}
\end{equation*}
and
\begin{equation*}
\norm{\widehat{\starhat}^{\rho_{2,s}, \rhohat_2, \rhohat_2}}_{W^{1,p}}\leq
\mathfrak{p}\left( C, \norm{\rho}_{W^{1,p}}, \norm{\rhohat_2}_{W^{1,p}}
\right)\norm{\rhohat_2}^{2}_{W^{1,p}}.
\end{equation*}
Further,
\begin{equation*}
\begin{split}
&\norm{\theta^{\rho_{1,s}} - \theta^{\rho_{2,s}}}_{W^{2,p}}
\leq
\norm{\int_0^{1} \frac{d}{ds'}\theta^{s' \rho_{1,s} + (1 - s')
\rho_{2,s}}\, ds'}_{W^{2,p}}\\
&\quad\leq
\sup_{s' \in [0,1]}\bigg(\norm{\thetahat(s' \rho_{1,s} + (1 - s')
\rho_{2,s})}_{L^\infty}\norm{\rhohat_1 - \rhohat_2}_{W^{2,p}}\\
&\qquad +
\norm{\thetahat(s' \rho_{1,s} + (1 - s')
\rho_{2,s})}_{W^{2,p}}\norm{\rhohat_1 - \rhohat_2}_{L^\infty}\bigg)\\
&\quad\leq
\mathfrak{p}\left( C, \norm{\rho}_{L^\infty}, \norm{\rhohat_1}_{L^\infty},
\norm{\rhohat_2}_{L^\infty}\right)\norm{\rhohat_1 -
\rhohat_2}_{W^{2,p}}\\
&\qquad +
\mathfrak{p}\left( C, \norm{\rho}_{L^\infty}, \norm{\rhohat_1}_{L^\infty},
\norm{\rhohat_2}_{L^\infty}\right)\norm{\rhohat_1 - \rhohat_2}_{L^\infty}\\
&\quad\quad
\cdot \sup_{s' \in [0,1]}\left(1 + \norm{s' \rho_{1,s} + (1 - s')
\rho_{2,s}}_{W^{2,p}} \left( 1 +
\norm{s' \rho_{1,s} + (1 - s')
\rho_{2,s}}_{L^\infty}\right)\right)\\
&\quad\leq
\mathfrak{p}\left( C, \norm{\rho}_{L^\infty}, \norm{\rhohat_1}_{L^\infty},
\norm{\rhohat_2}_{L^\infty}\right)\norm{\rhohat_1 -
\rhohat_2}_{W^{2,p}}\\
&\qquad +
\mathfrak{p}\left( C, \norm{\rho}_{L^\infty}, \norm{\rhohat_1}_{L^\infty},
\norm{\rhohat_2}_{L^\infty}\right) \norm{\rhohat_1 -
\rhohat_2}_{L^\infty}\\
&\qquad\qquad
\cdot \left(\norm{\rho}_{W^{2,p}} + \norm{\rhohat_1}_{W^{2,p}} +
\norm{\rhohat_2}_{W^{2,p}}\right)\\
\end{split}
\end{equation*}
Combining these inequalities yields
\begin{equation*}
\begin{split}
&\norm{\widehat{\starhat}^{\rho_{1,s},\rhohat_1, \rhohat_1}
d\theta^{\rho_{1,s}} - \widehat{\starhat}^{\rho_{2,s},\rhohat_2, \rhohat_2}
d\theta^{\rho_{2,s}}}_{W^{1,p}}\\
&\qquad\leq
\mathfrak{p}\left(C, \norm{\rho}_{W^{1,p}},
\norm{\rhohat_1}_{W^{1,p}}, \norm{\rhohat_2}_{W^{1,p}}\right) \left(
\norm{\rhohat_1}_{W^{1,p}} + \norm{\rhohat_2}_{W^{1,p}}
\right)
\norm{\rhohat_1 - \rhohat_2}_{W^{2,p}}\\
&\quad\qquad +
\mathfrak{p}\left(C, \norm{\rho}_{W^{1,p}},
\norm{\rhohat_1}_{W^{1,p}}, \norm{\rhohat_2}_{W^{1,p}}\right)\left(
\norm{\rho}_{W^{2,p}} + \norm{\rhohat_1}_{W^{2,p}} +
\norm{\rhohat_2}_{W^{2,p}}
\right)\\
&\qquad \qquad \qquad
\cdot \left(\norm{\rhohat_1}_{W^{1,p}} +
\norm{\rhohat_2}_{W^{1,p}}\right)\norm{\rhohat_1 - \rhohat_2}_{W^{1,p}}
\end{split}
\end{equation*}
The terms~\eqref{eq:quadTerm2} and~\eqref{eq:quadTerm3} can be estimated with the same
techniques. This proves statement (ii) and the lemma.
\end{proof}

We will need the following lemma to use the technique of freezing the
coefficients for the operator $ d\frac{d^{*^\rho}}{u}: d\Om^1 \to d \Om^1$.
\begin{lemma}[{\bf Freezing the Coefficients}]\label{lem:freezingTheCoefficients}
Let $p>4$. There exists a polynomial with positive coefficients
$\mathfrak{p}$ with the following significance. Let $\rho_0 \in B^{1,p}_2(M,
\Lambda^2)$ be a nondegenerate 2-form and let
$$
\rhotilde_0 \in W^{1,2}(I, L^p(M,\Lambda^2)) \cap L^2(I, W^{2,p}(M,
\Lambda^2))
$$
be a path of nondegenerate 2-forms such that it is a continous extension of
$\rho_0$ with
$\rhotilde_0(t=0,\cdot) = \rho_0$. Then
\begin{multline*}
\norm{(d\frac{d^{*^{\rhotilde_0}}}{\tilde{u}_0} - d\frac{d^{*^{\rho_0}}}{u_0}
)\rhohat}_{L^2(I, L^p)}\\
\leq
\mathfrak{p}(C, \norm{\rhotilde_0}_{L^\infty(I, W^{1,p})},
\norm{\rho_0}_{W^{1,p}}, T) \norm{\rhohat}_{L^2(I, W^{2,p})}
\norm{\rhotilde_0 - \rho_0}_{L^\infty(I, W^{1,p})},
\end{multline*}
where
$$
C:= \sup_{\substack{(t,x)\in I\times M \\ s\in [0,1]}} \frac{1}{u_{s
\rhotilde_0(t,x) + (1-s) \rho_0}},\qquad I :=
[0, T).
$$
\end{lemma}
\begin{proof}
For ease of notation we don't number the different polynomials appearing in
the following estimates. Let $\rhohat\in d\Om^2(M)$.  In a local
coordinate chart the expression $ d\frac{d^{*^{\rho_0}}}{u_0} \rhohat$ is
given by
$$
\frac{1}{\sqrt{\Abs{\det g^{\rho_0}}}} \sum_{a < b}\sum_{ij} \p_i
\left(\frac{\sqrt{\Abs{\det g^{\rho_0}}}}{u_0} (g^{\rho_0})^{ij}_{ab}\p_j
\rhohat_{ab}\right)  dx^a \wedge dx^b,
$$
where $(g^{\rho_0})^{cd}_{ab}$ are the coefficients of the inverse of the
metric on the two-forms induced by $g^{\rho_0}$. Let
\begin{gather*}
f_{\rho_0} :=\frac{1}{\sqrt{\Abs{\det g^{\rho_0}}}}, \qquad
h_{\rho_0}:=\frac{\sqrt{\Abs{\det g^{\rho_0}}}}{u_0} (g^{\rho_0})^{ij}_{ab}.
\end{gather*}
Let $\rhotilde_0$ be a continuous extension of $\rho_0$ in $W^{1,2}(I,
L^p) \cap L^2(I, W^{2,p})$. For a fixed $t\in I$ we estimate
\begin{equation*}
\begin{split}
&\norm{f_{\rhotilde_0} \p_i \left(h_{\rhotilde_0} \p_j \rhohat_{ab}\right)
- f_{\rho_0} \p_i \left(h_{\rho_0} \p_j \rhohat_{ab}\right)}_{L^p}\\
&\quad\leq
\norm{\left(f_{\rhotilde_0} - f_{\rho_0}\right)\p_i \left(h_{\rhotilde_0}
\p_j \rhohat_{ab}\right) }_{L^p}
+ \norm{f_{\rho_0} \left(\p_i \left(h_{\rhotilde_0} \p_j \rhohat_{ab}\right)
- \p_i \left(h_{\rho_0} \p_j \rhohat_{ab}\right)\right)}_{L^p}\\
\end{split}
\end{equation*}
Then
\begin{equation*}
\begin{split}
&\norm{\left(f_{\rhotilde_0} - f_{\rho_0}\right)\p_i
\left(h_{\rhotilde_0} \p_j \rhohat_{ab}\right)}_{L^p}\\
&\qquad\leq
\norm{f_{\rhotilde_0} - f_{\rho_0}}_{L^\infty} \norm{
\left(h_{\rhotilde_0} \p_j \rhohat_{ab}\right)}_{W^{1,p}}\\
&\qquad\leq
\norm{f_{\rhotilde_0} - f_{\rho_0}}_{L^\infty} \norm{
h_{\rhotilde_0}}_{W^{1,p}} \norm{\rhohat_{ab}}_{W^{2,p}}\\
&\qquad\leq
\fp\left(C, \norm{\rhotilde_0}_{L^\infty},
\norm{\rho_0}_{L^\infty}\right) \norm{\rhotilde_0 - \rho_0}_{L^\infty}
\norm{\rhotilde_0}_{W^{1,p}}\norm{\rhohat}_{W^{2,p}}
\end{split}
\end{equation*}
where
$$
C:= \sup_{\substack{(t,x)\in [0,T)\times M \\ s\in [0,1]}} \frac{1}{u_{s
\rhotilde_0(t,x) + (1-s) \rho_0}}
$$
and $\mathfrak{p}$ is a polynomial with positive coefficients in the
indicated variables. In the same way we find that
\begin{multline*}
\norm{f_{\rho_0} \left(\p_i \left(h_{\rhotilde_0} \p_j \rhohat_{ab}\right)
- \p_i \left(h_{\rho_0} \p_j \rhohat_{ab}\right)\right)}_{L^p}\\
\leq
\mathfrak{p}\left( C, \norm{\rhotilde_0}_{W^{1,p}},
\norm{\rho_0}_{W^{1,p}}
\right)\norm{h_{\rhotilde_0} -
h_{\rho_0}}_{W^{1,p}}\norm{\rhohat}_{W^{2,p}}.
\end{multline*}
Hence,
\begin{equation*}
\begin{split}
&\norm{f_{\rhotilde_0} \p_i \left(h_{\rhotilde_0} \p_j \rhohat_{ab}\right)
- f_{\rho_0} \p_i \left(h_{\rho_0} \p_j \rhohat_{ab}\right)}_{L^2(I,
L^p)}\\
&\qquad\leq
\mathfrak{p}\left( C, \norm{\rhotilde_0}_{L^\infty(I, W^{1,p})},
\norm{\rho_0}_{W^{1,p}}, T \right)\\
&\qquad\qquad \cdot
\norm{\rhohat}_{L^2(I, W^{2,p})} \left(\norm{f_{\rhotilde_0} -
f_{\rho_0}}_{L^\infty(I, W^{1,p})} + \norm{h_{\rhotilde_0} -
h_{\rho_0}}_{L^\infty(I, W^{1,p})}\right).
\end{split}
\end{equation*}
The functions $f_\rho$ and $h_\rho$ depend smoothly on the point wise values
of $\rho$, in fact they are polynomials in the variables $\rho$ and
$\frac{1}{u}$. Therefore there exists a polynomial $\mathfrak{p}$ such that
\begin{multline*}
\norm{f_{\rhotilde_0} - f_{\rho_0}}_{L^\infty(I, W^{1,p})} +
\norm{h_{\rhotilde_0} - h_{\rho_0}}_{L^\infty(I, W^{1,p})}\\
\leq
\mathfrak{p}\left( C,
\norm{\rho_0}_{W^{1,p}}\norm{\rhotilde_0}_{L^\infty(I, W^{1,p})} \right)
\norm{\rhotilde_0 - \rho_0}_{L^\infty(I, W^{1,p})}.
\end{multline*}
This proves the lemma.
\end{proof}
The following lemma gives an estimate for the lower order term appearing in the
linearization of the gradient.
\begin{lemma}[{\bf Lower Order Terms}]\label{lem:lowerOrderTerms} Let $p>4$.
There exist polynomials $\mathfrak{p}_1, \mathfrak{p}_2$ with positive
coefficients with the following significance. Let $\rho\in W^{1,2}(I,
L^p(M, \Lambda^2)) \cap L^2(I, W^{2,p}(M, \Lambda^2))$ be a path of nondegenerate two-forms and let
$A^{\rho}: d\Om^1 \to d\Om^1$ be the operator defined by~\eqref{eq:Arho}.
Then for all $\rhohat\in W^{1,2}(I, L^p(M, \Lambda^2)) \cap L^2(I,
W^{2,p}(M, \Lambda^2))$,
\begin{equation*}
\begin{split}
\norm{A^\rho \rhohat}_{L^2(I,L^p)}
&\leq
\bigg(T^{\frac12} \mathfrak{p}_1\left( C, \norm{\rho}_{L^\infty(I,
W^{1,p})} \right)\\
&\qquad
+ \mathfrak{p}_2\left( C, \norm{\rho}_{L^\infty(I, W^{1,p})} \right)
\norm{\rho}_{L^2(I,W^{2,p})}\bigg)\norm{\rhohat}_{L^\infty(I, W^{1,p})},
\end{split}
\end{equation*}
where
$$
I := [0, T), \qquad C:= \sup_{(t,x)\in I\times M} \frac{1}{u(t,x)},\qquad u
:= \frac{\rho\wedge\rho}{2\dvol}.
$$
\end{lemma}
\begin{proof}
For a fixed $t\in I$ we can estimate
\begin{equation*}
\begin{split}
\norm{A^\rho \rhohat}_{L^p}
&\leq
\norm{*^\rho\left(
  \left(\frac{du}{u^2}\right)\wedge \left( \rhohat + *^\rho \rhohat
\right)\right)}_{W^{1,p}} + \norm{
  *^\rho \left(d\left( \Abs{\frac{\rho^+}{u}}^2 \right) \wedge
\rhohat\right)}_{W^{1,p}}\\
&\qquad +
\norm{
  \hat{*}^{\rho,\rhohat} d\theta^\rho}_{W^{1,p}}\\
  &\leq
  \mathfrak{p}\left( C, \norm{\rho}_{W^{1,p}} \right)\left(1 +
  \norm{\rho}_{W^{2,p}}\left( 1 +
  \norm{\rho}_{L^\infty} \right)\right)\norm{\rhohat}_{W^{1,p}}
\end{split}
\end{equation*}
using Lemma~\ref{lem:sobolevProd}. Hence,
\begin{equation*}
\begin{split}
  \norm{A^\rho \rhohat}_{L^2(I,L^p)}
  &\leq
  \bigg(T^{\frac12} \mathfrak{p}_1\left( C, \norm{\rho}_{L^\infty(I,
    W^{1,p})} \right)\\
    &\qquad
    + \mathfrak{p}_2\left( C, \norm{\rho}_{L^\infty(I, W^{1,p})} \right)
    \norm{\rho}_{L^2(I,W^{2,p})}\bigg)\norm{\rhohat}_{L^\infty(I, W^{1,p})},
  \end{split}
\end{equation*}
for two polynomials $\mathfrak{p}_1$, $\mathfrak{p}_2$ with positive
coefficients. This proves the lemma.
\end{proof}
We are now ready to give a proof of the local existence
Theorem~\ref{thm:localExistence}.
\begin{proof}[Proof of Theorem \ref{thm:localExistence}]
\label{proof:localExistence}
Define $\rhotilde_0$ by~\eqref{eq:homogenousCauchy} and $f$
by~\eqref{eq:steIteration}. It follows with Lemma~\ref{lem:sobolevProd} that
for a fixed $t\in I$, there exists a polynomial with positive coefficients
$\mathfrak{p}$ such that
\begin{equation*}
  \begin{split}
    \norm{d*^{\rhotilde_0} d \theta^{\rhotilde_0}}_{L^p}
    &\leq
    \norm{*^{\rhotilde_0} d \theta^{\rhotilde_0}}_{W^{1,p}}\\
    &\leq
    \mathfrak{p}\left( C, \norm{\rhotilde_0}_{W^{1,p}} \right) \left( 1 +
    \norm{\rhotilde_0}_{W^{2,p}}\left( 1 + \norm{\rhotilde_0}_{L^\infty}
    \right) \right)
  \end{split}
\end{equation*}
where
$$
C := \sup_{(t,x)\in I\times M}\frac{1}{u_0(t,x)}, \qquad \tilde{u}_0 =
\frac{\rhotilde_0\wedge \rhotilde_0}{2\dvol}.
$$
Hence,
\begin{multline*}
  \norm{d*^{\rhotilde_0} d \theta^{\rhotilde_0}}_{L^2(I,L^p)}\\
  \leq
  T^{\frac12}\mathfrak{p}_1\left( C, \norm{\rhotilde_0}_{L^\infty(I,
    W^{1,p})} \right) + \mathfrak{p}_2\left( C,
    \norm{\rhotilde_0}_{L^\infty(I, W^{1,p})} \right)
    \norm{\rhotilde_0}_{L^2(I, W^{2,p})}.
  \end{multline*}
  Together with the quadtratic estimates of Lemma~\ref{lem:quadraticEstimate}
  and Lemmas~\ref{lem:freezingTheCoefficients} and~\ref{lem:lowerOrderTerms} it
  follows that there exists polynomials with positive coefficients $\fp_1 \ldots
  \fp_7$ with the following significance,
  \begin{equation}\label{eq:fEstimate}
    \begin{split}
      &\norm{f}_{L^2(I, L^p)} \leq\\
      &\quad
      \fp_1\left( C, \norm{\rhotilde_0}_{L^\infty(I, W^{1,p})},
      \norm{\rhohat}_{L^\infty(I, W^{1,p})} \right) \norm{\rhohat}_{L^\infty(I,
        W^{1,p})}\norm{\rhohat}_{L^2(I, W^{2,p})}\\
        &\quad+
        \fp_2\left( C, \norm{\rhotilde_0}_{L^\infty(I, W^{1,p})},
        \norm{\rhohat}_{L^\infty(I,W^{1,p})}
        \right)\norm{\rhotilde_0}_{L^2(I,W^{2,p})}\norm{\rhohat}_{L^\infty(I,
          W^{1,p})}^2\\
          &\quad +
          T^\frac12\fp_3 \left( C, \norm{\rhotilde_0}_{L^\infty(I,W^{1,p})}
          \right)
          +
          \fp_4(C, \norm{\rhotilde_0}_{L^\infty(I,
            W^{1,p})})\norm{\rhotilde_0}_{L^2(I, W^{2,p})}\\
            &\quad +
            \left(T^\frac12\fp_5\left( C, \norm{\rhotilde_0}_{L^\infty(I,
              W^{1,p})} \right) + \fp_6\left( C, \norm{\rhotilde_0}_{L^\infty(I,
                W^{1,p})} \right)\norm{\rhotilde_0}_{L^2(I, W^{2,p})}
                \right)\\
                &\qquad \qquad
                \cdot \norm{\rhohat}_{L^\infty(I, W^{1,p})}\\
                &\quad+
                \fp_7\left( C, \norm{\rhotilde_0}_{L^\infty(I, W^{1,p})},
                \norm{\rho_0}_{W^{1,p}}, T \right)\norm{\rhohat}_{L^2(I,
                  W^{2,p})}\norm{\rhotilde_0 - \rho_0}_{L^\infty(I, W^{1,p})},
                \end{split}
              \end{equation}
              where
              $$
              C:=\sup_{\substack {(t,x)\in I\times M \\ 0 \leq s \leq 1}}
              \frac{1}{u_{\rhotilde_0(t,x) + s \rhohat(t,x)}},\qquad u_\rho :=
              \frac{\rho\wedge \rho}{2\dvol}.
              $$
              Since $\rhotilde_0$ solves the homogenous Cauchy
              problem~\eqref{eq:homogenousCauchy},
              $$
              \norm{\rhotilde_0}_{L^2(I, W^{2,p})} \leq  \norm{\rho_0}_{B^{1,p}_2}
              $$
              and by choosing $T$ small we can get
              $\norm{\rhotilde_0}_{L^2(I, W^{2,p})}$ arbitrary small. Also by choosing $T$
              small enough, we can garantee that $\norm{\rhotilde_0 - \rho_0}_{L^\infty(I,
                W^{1,p})}$ is small and that $\rhotilde_0$ is nondegenerate for all $t\in
                I$. By choosing $r>0$ small, we can garantee that $C < \infty$, since
                $\norm{\rhohat}_{L^\infty(I, L^\infty)}\leq r$. It then follows
                from~\eqref{eq:fEstimate} that by choosing $r>0, T>0$ small we get
                that
                $$
                \norm{\rhohat'}_{L^2(I,W^{2,p})} + \norm{\rhohat'}_{L^\infty(I, W^{1,p})} +
                \norm{\p_t \rhohat'}_{L^2(I, L^p)} \leq \norm{f}_{L^2(I, L^p)} \leq r
                $$
                and in particular $\sR(B_r) \subseteq B_r$.

                Now let $\rhohat_1,\rhohat_2 \in B_r$ and
                $$
                \rhohat_1' := \sR\left( \rhohat_1 \right),\qquad \rhohat_2' := \sR\left(
                \rhohat_2 \right).
                $$
                Then
                $$
                \p_t\left( \rhohat_1' - \rhohat_2' \right) + d\frac{d^{*^{\rho_0}}}{u_0}\left(
                \rhohat_1' - \rhohat_2' \right) = f\left( \rhohat_1 \right) - f\left(
                \rhohat_2\right).
                $$
                where
                \begin{equation*}
                  \begin{split}
                    f\left( \rhohat_1 \right) - f\left( \rhohat_2\right) &=
                    \left(d*^{\rho_1}d\theta^{\rho_1} - d*^{\rho_2}d\theta^{\rho_2} +
                    L_{\rhotilde_0} \left( \rhohat_1 - \rhohat_2 \right)\right)
                    \\
                    &\qquad
                    + \left(-A^{\rhotilde_0} + d\frac{d^{*^{\rho_0}}}{u_0} -
                    d\frac{d^{*^{\rhotilde_0}}}{\tilde{u}_0} \right) \left( \rhohat_1 -
                    \rhohat_2 \right)
                  \end{split}
                \end{equation*}
                and
                \begin{gather*}
                  \rho_1 := \rhotilde_0 + \rhohat_1,
                  \qquad
                  \rho_2 := \rhotilde_0 +
                  \rhohat_2,
                  \qquad u_0 := \frac{\rho_0 \wedge \rho_0}{2\dvol},
                  \qquad
                  \tilde{u}_0 := \frac{\rhotilde_0\wedge \rhotilde_0}{2\dvol}.
                \end{gather*}
                It follows from Lemma~\ref{lem:quadraticEstimate} (ii),
                Lemma~\ref{lem:lowerOrderTerms} and Lemma~\ref{lem:freezingTheCoefficients}
                that there exist polynomials with positive coefficients $\fp_1, \ldots, \fp_5$
                with the following significance,
                \begin{equation*}
                  \begin{split}
                    &\norm{f\left( \rhohat_1 \right) - f\left( \rhohat_2 \right)}_{L^2(I,
                    L^p)}\\
                    &\quad\leq
                    \fp_1\bigg( C, \norm{\rhohat_1}_{L^\infty(I, W^{1,p})},
                    \norm{\rhohat_2}_{L^\infty(I, W^{1,p})}, \norm{\rhohat_1}_{L^2(I,
                      W^{2,p})},\norm{\rhohat_2}_{L^2(I,
                        W^{2,p})},\\
                        &\qquad \qquad \qquad
                        \norm{\rhotilde_0}_{L^2(I, W^{2,p})} \bigg) \cdot
                        r \cdot \left(\norm{\rhohat_1 -
                        \rhohat_2}_{L^\infty(I, W^{1,p})} + \norm{\rhohat_1 -
                        \rhohat_2}_{L^2(I, W^{2,p})}\right)\\
                        &\quad +
                        \fp_3\left( C, \norm{\rhotilde_0}_{L^\infty(I, W^{1,p})},
                        \norm{\rho_0}_{W^{1,p}},T \right)\norm{
                        \rhotilde_0 - \rho_0}_{L^\infty(I, W^{1,p})} \norm{\rhohat_1 -
                        \rhohat_2}_{L^2(I, W^{2,p})}\\
                        &\quad +
                        \bigg( T^{\frac12}\fp_4(C,\norm{\rhotilde_0}_{L^\infty(I, W^{1,p})} +
                        \fp_5\left( C, \norm{\rhotilde_0}_{L^\infty(I, W^{1,p})} \right)
                        \norm{\rhotilde_{0}}_{L^2(I, W^{2,p})}\bigg)\\
                        &\qquad \qquad \qquad \cdot
                        \norm{\rhohat_1 - \rhohat_2}_{L^\infty(I, W^{1,p})},
                      \end{split}
                    \end{equation*}
                    where
                    $$
                    C :=  \sup_{\left( t,x \right) \in I \times M} \frac{1}{\tilde{u}_0\left( t,x
                  \right)} +
                  \sup_{\left( t,x \right) \in I \times M} \frac{1}{u_{\rho_0}(x) + s \left(
                s'\rhohat_1 (t, x) + \left( 1 - s' \right)\rhohat_2(t,x) \right)}
                $$
                Therefore we can choose $r$ and $T$ small such that
                \begin{multline*}
                  \norm{\rhohat_1' - \rhohat_2'}_{L^2(I, W^{2,p})} + \norm{\p_t(\rhohat_1' -
                  \rhohat_2')}_{L^2(I, L^p)}\\
                  \leq \frac12 \left(\norm{\rhohat_1 - \rhohat_2}_{L^\infty(I,
                    W^{1,p})} + \norm{\rhohat_1 - \rhohat_2}_{L^2(I, W^{2,p})}\right)
                  \end{multline*}
                  Thus, $\sR$ is a contraction on $B_r$. The theorem now follows from Banach's
                  fixed point theorem.
                \end{proof}
                \begin{remark}
                  If one would use the more heavy machinery of parabolic $L^p - L^q$ theory,
                  one would find with the same line of arguments that there exists a solution
                  to the Donaldson flow for short times in the space
                  $$
                  L^q(I, W^{2,p}) \cap W^{1,q}(I, L^p)
                  $$
                  for the initial condition
                  $$
                  \rho(t=0, \cdot) \in B^{s,p}_{q}, \quad s:= 2 - \frac{2}{q}
                  $$
                  for $sp > 4$. Here, $B^{s,p}_{q}$ denotes the Besov space~\cite{Grigoryan}.
                  This is the lowest possible condition at the regularity of the initial
                  condition. It asserts that solutions are continuous in time and space.
                \end{remark}



\section{Semiflow}\label{sec:semiflow}
We prove that the Donaldson flow is a local semiflow on the Besov space
$B^{1,p}_2(M, \Lambda^2)$ restricted to symplectic forms representing a given
cohomology class, i.e. for every symplectic form in $B^{1,p}_2(M,\Lambda^2)$
representing the cohomology class $a \in H^2(M;\R)$ there is an open
neighborhood $\sU \subset B^{1,p}_{2,a}(M, \Lambda^2)$ of this element and a
$T>0$ such that the map
$$
[0, T] \times \sU \to \sS^{1,p}_{a,2}(M, \Lambda^2)
: \qquad (t,\rho_0) \mapsto \rho(t,\cdot)
$$
is smooth, where
$$
\sS^{1,p}_{a,2} := \left\{ \rho \in B^{1,p}_2(M,\Lambda^2) | d\rho =0,\,
\rho\wedge\rho >0,\, [\rho] = a\right\}.
$$
and $\rho$ is the unique solution to the Donaldson flow with initial condition
$\rho(t=0,\cdot) = \rho_0$. Our argumentation for this result follows the one
given in~\cite{ANGENENT}, section 2.2.

\begin{theorem}[{\bf Semiflow}]\label{thm:semiflow}
The Donaldson flow is a local semiflow on\\
$\sS^{1,p}_{2,a}(M, \Lambda^2)$.
\end{theorem}
\begin{proof}
  See page~\pageref{proof:semiflow}.
\end{proof}
For the proof of the theorem we need the following two lemmas.  Let $p>4$. Let
$I = [0,T]$ be a closed interval of the reals.  Let
$$
\bar{\rho} \in L^2(I, W^{2,p}(M, \Lambda^2)) \cap W^{1,2}(I, L^p(M,
\Lambda^2))
$$
be a path of nondegenerate two-forms and let $L_{\bar{\rho}}$ be the operator
defined by equation~\eqref{eq:Lrho}. Define the spaces
$$
\sXhat := \{\rhohat \in L^2(I, W^{2,p}(M, \Lambda^2)) \cap W^{1,2}(I,
L^p(M, \Lambda^2)) |\ \forall t: \rhohat(t,\cdot)\ \text{exact}\}.
$$
and
$$
\sYhat:= \{ \etahat \in L^2(I, L^p(M, \Lambda^2)) |\ \forall t:
\etahat(t,\cdot) \ \text{exact}\}
$$
and consider the operator
$$
\p_t + L_{\bar{\rho}} : \sXhat \to \sYhat.
$$
We'll also need the following $L^2$ versions of the spaces $\sXhat,\sYhat$,
$$
\sXhat_2 := \{\rhohat \in L^2(I, W^{2,2}(M, \Lambda^2)) \cap W^{1,2}(I,
L^2(M, \Lambda^2)) |\ \forall t: \rhohat(t,\cdot)\ \text{exact}\}.
$$
and
$$
\sYhat_2:= \{ \etahat \in L^2(I, L^2(M, \Lambda^2)) |\ \forall t:
\etahat(t,\cdot) \ \text{exact}\}
$$
The next lemma shows that the image of the operator $\p_t + L_\rhobar$ for a
smooth $\rhobar$ restricted to $\sXhat_2$ is dense in $\sYhat_2$.
\begin{lemma}[{\bf Dense Image}]\label{lem:denseImage}
  Let $\rhobar$ be a smooth path of smooth nondegenerate two-forms. There
  exists a constant $c_0>0$ such that for all $c>c_0$ the following holds.
  Assume $\etahat \in \sYhat_2$ such that
  \begin{equation}\label{eq:denseImage}
    \int_I \int_M g^{\rhobar}\left(\p_t \phihat + L_\rhobar \phihat + c
    \phihat,
    \etahat\right) \dvol dt= 0
  \end{equation}
  for all exact $\phihat \in C_0^\infty(I, C^\infty_0(M, \Lambda^2))$. Then
  $\etahat = 0$.
\end{lemma}
\begin{proof}
  We follow the proof of Theorem~3.10 in~\cite{ROBSAL}. The proof has 5 steps.
  Let $B^{\rhobar_t} \in \Gamma(M, \Aut(\Lambda^2))$ be the automorphism valued
  section of $M$ defined by the equation
  $$
  g^{\rhobar_t}(\eta, \sigma) =: g(\eta, B^{\rhobar_t} \sigma)
  $$
  for each fixed $t \in I$ and two-forms $\eta, \sigma$, where $g$ is the
  background metric of $M$ and $g^\rhobar$ is as described in the introduction.

  \smallskip\noindent\emph{Step 1.} $\etahat \in W^{1,2}(I,
  (W^{2,2}(\Lambda^2))^*)$ and
  \begin{equation}\label{eq:derivativeEtahat}
    \p_t \etahat - L_\rhobar^{*^\rhobar}\etahat - c \etahat+ (B^{\rhobar})^{-1}(\p_t B^{\rhobar})
     \etahat = 0,
  \end{equation}
  where $L_\rhobar^{*^\rhobar} : \sYhat_2 \to L^2(I, (W^{2,2}(M,
  \Lambda^2))^{*})$ is
  the formal adjoint of $L_\rhobar$ with respect to the inner product
  \eqref{eq:metric}.

  For $\phihat \in C_0^{\infty}(I,
  C_0^\infty(M, \Lambda^2))$
  \begin{align*}
    &\int_I \int_M g(\p_t \phihat,B^\rhobar \etahat)\dvol dt\\
    &=
    \int_I \int_M g^\rhobar(\p_t \phihat, \etahat)\dvol dt\\
    &=
    - \int_I\int_M g^\rhobar(L_\rhobar\phihat + c\phihat, \etahat) \dvol dt\\
    &=
    - \int_I\int_M g^\rhobar(\phihat, L^{*^\rhobar}_\rhobar\etahat + c\etahat) \dvol dt\\
    &=
    - \int_I\int_M g\left(\phihat, B^\rhobar \left(L^{*^\rhobar}_\rhobar\etahat +
    c\etahat\right)\right) \dvol dt\\
    &=
    - \int_I\int_M g\left(\int_0^t\p_s \phihat(s,\cdot)\,
    ds, B^\rhobar \left(L^{*^\rhobar}_\rhobar\etahat + c \etahat\right)\right) \dvol\, dt\\
    &=
    - \int_I\int_t^T\int_M g\left(\p_s \phihat(s,\cdot)
    , \left(B^\rhobar \left(L^{*^\rhobar}_\rhobar\etahat + c \etahat
    \right)\right) (t)\right) \dvol\, dt ds\\
    &=
    - \int_I\int_M g\left(\p_t \phihat
    , \int_t^T \left(B^\rhobar \left(L^{*^\rhobar}_\rhobar\etahat + c \etahat
    \right)\right)(s) ds\right) \dvol\,
    dt.
  \end{align*}
  Since the time derivatives of test functions $\phi \in C^{\infty}_0(I,
  C_0^\infty(M, \Lambda^2))$ are dense in $L^2(I, L^2(M, \Lambda^2))$ it follows
  that
  $$
  B^\rhobar \etahat = - \int_t^T \left(B^\rhobar\left(
  L^{*^\rhobar}_\rhobar\etahat + c \etahat\right)\right)(s) ds
  $$
  and therefore
  $$
  (\p_t  B^\rhobar) \etahat + B^\rhobar \p_t \etahat = B^\rhobar \left(
  L^{*^\rhobar}_\rhobar\etahat + c \etahat\right).
  $$
  This proves Step~1.

  For ease of notation we define
  \begin{gather*}
    L'_\rhobar := L^{*^\rhobar}_\rhobar + c\, \id - (B^\rhobar)^{-1} (\p_t
    B^\rhobar) \\
    \Ltilde_{\rhobar}(t,\cdot) := L'_{\rhobar(T -t,\cdot)}.
  \end{gather*}
  Now choose a smooth cutoff function $\theta: \R \to \R$ such that $\theta(t)
  = 0$ for $t \in \R \setminus [0,1]$, $\theta(t) \geq 0$ for all $t \in \R$ and $\int
  \theta = 1$. For $\delta >0$ define $\theta_\delta(t) := \delta^{-1} \theta
  (\delta^{-1} t)$.

  \smallskip\noindent\emph{Step 2.}
  Extend $\etahat$ with zero outside of $I$ and define
  $\etatilde(t,\cdot) :=\etahat(T-t,\cdot) $. For $\delta > 0$ sufficiently
  small we have
  $$
  \etatilde_\delta := \theta_\delta * \etatilde \in L^2(\R, W^{2,2}(M, \Lambda^2))
  \cap W^{1,2}(\R, L^2(M, \Lambda^2)).
  $$

  First note that the highest order term of $\Ltilde_{\rhobar_t}$
  equals
  $$
  d\frac{d^{*^{\rhobar_{T- t}}}}{u_{T -t}} : \{ \rhohat \in W^{2,2}(M, \Lambda^2)\, | \,
  \rhohat \text{ exact}\} \to \{ \etahat \in L^2(M, \Lambda^2) \, | \,
    \etahat \text{ exact}\}
  $$
  for
  a fixed $t$ (where $\rhobar_t$ is defined). If $c>c_0>0$ is big enough,
  $\Ltilde_\rhobar$ is positive definite on the exact two-forms in $W^{1,2}(M,
  \Lambda^2)$ and it follows from the Lax-Milgram theorem and the elliptic
  regularity Lemma~\ref{lem:ellipticRegularity} that this operator is
  invertible on the exact two-forms in $L^2(M, \Lambda^2)$. By Step~1,
  $$
  \p_t \etatilde (t,\cdot)= - \p_t \etahat (T- t,\cdot) = -(L'_{\rhobar}
  \etahat) (T - t,\cdot)=  - \Ltilde_{\rhobar}(t,\cdot)
  \etatilde(t,\cdot).
  $$
  We multiply this equation with
  $-(\Ltilde_{\rhobar})^{-1}$ to find
  $$
  \etatilde = -(\Ltilde_{\rhobar})^{-1} \p_t \etatilde.
  $$
  The convolution with $\theta_\delta$ is then given by
  \begin{align*}
    \etatilde_\delta
    &=
    -\theta_\delta * \left((\Ltilde_{\rhobar})^{-1} \p_t \etatilde\right)
    \\
    &=
    -\left(\p_t \theta_\delta\right) * \left((\Ltilde_{\rhobar})^{-1}
    \etatilde\right) - \theta_\delta *\left((\Ltilde_{\rhobar})^{-1} (\p_t
    \Ltilde_{\rhobar})(\Ltilde_{\rhobar})^{-1}\etatilde \right)\\
    &=
    -\left(\p_t \theta_\delta\right) * \left((\Ltilde_{\rhobar})^{-1}
    \etatilde\right) + \theta_\delta *\left((\Ltilde_{\rhobar})^{-1}
    \zetatilde\right)
  \end{align*}
  where
  $$
  \zetatilde := -(\Ltilde_\rhobar)^{-1}(\p_t
  \Ltilde_\rhobar)(\Ltilde_{\rhobar})^{-1}\etatilde
  $$
  and we're using the identity $\theta * (u \p_t v) = (\p_t \theta) * (uv) -
  \theta * ((\p_t u)v)$.
  This proves Step~2.

  \smallskip\noindent\emph{Step 3.}
  There exists a constant $c>0$ independent of $\delta$ such that
  $$
  \norm{\p_t \etatilde_\delta + \Ltilde_\rhobar \etatilde_\delta}_{L^2(\R, L^2)}
  \leq c
  $$
  for all sufficiently small $\delta > 0$.

  With Step~2 and the identity $\p_t \etatilde_\delta = (\p_t \theta_\delta) *
  \etatilde$ it follows that
  \begin{align*}
    &\p_t \etatilde_\delta + \Ltilde_\rhobar\etatilde_\delta\\
    &\qquad=
    (\p_t \theta_\delta) * \etatilde - \Ltilde_\rhobar
    \left(\p_t \theta_\delta\right) * \left((\Ltilde_{\rhobar})^{-1}
    \etatilde\right) + \Ltilde_\rhobar \theta_\delta
    *\left((\Ltilde_{\rhobar})^{-1} \zetatilde\right)\\
    &\qquad=
    -\Ltilde_\rhobar \left(-\Ltilde_\rhobar^{-1 }(\p_t \theta_\delta) * \etatilde
    + \left(\p_t \theta_\delta\right) * \left((\Ltilde_{\rhobar})^{-1}
    \etatilde\right)\right) + \Ltilde_\rhobar \theta_\delta
    *\left((\Ltilde_{\rhobar})^{-1} \zetatilde\right).
  \end{align*}
  Since
  $$
  \rhobar \in L^{2}(I, W^{2,p}(M,\Lambda^2)) \cap W^{1,2}(I, L^p(M,
  \Lambda^2)) \hookrightarrow C^0(I, W^{1,p}(M, \Lambda^2)
  $$
  the operator norms of $L_\rhobar, \Ltilde_\rhobar$ are bounded on $I$ and so
  is the norm of $\zetahat$. Therefore the second term on the right is bounded
  in $L^2(\R, L^2(M, \Lambda^2))$, uniformly in $\delta$. For the first term
  on the right hand side we have the estimate
  \begin{align*}
    &\norm{\left(\Ltilde_\rhobar^{-1 }(\p_t \theta_\delta) * \etatilde -
    (\left(\p_t \theta_\delta\right) * \left((\Ltilde_{\rhobar})^{-1}
    \etatilde\right)\right)(t)}_{W^{2,2}}\\
    &\qquad\leq
    \norm{\int_{t-\delta}^{t+\delta}\frac{1}{\delta}\dot{\theta}\left(\frac{t-s}{\delta}\right)
    \frac{\Ltilde_\rhobar^{-1}(t) - \Ltilde_\rhobar^{-1}(s)}{\delta}
    \etatilde(s)\, ds}_{W^{2,2}}\\
    &\qquad\leq
    c \int_\R\Abs{\frac{1}{\delta}\dot{\theta}\left(\frac{t-s}{\delta}\right)}
    \norm{\etatilde(s)}_{L^2}\, ds
  \end{align*}
  Here we're using that $\Ltilde^{-1}_{\rhobar}(s) = L'^{-1}_{\rhobar(T-s)}
  $ depends smoothly on $\rhobar(T - s)$.
  Then it follows with Young's inequality that
  \begin{multline*}
    \norm{-\Ltilde_\rhobar \left(-\Ltilde_\rhobar^{-1 }(\p_t \theta_\delta) * \etatilde +
    (\left(\p_t \theta_\delta\right) * \left((\Ltilde_{\rhobar})^{-1}
    \etatilde\right)\right)}_{L^2(\R, L^2)}\\
    \leq c \norm{\p_t \theta}_{L^1(\R)}
    \norm{\etatilde}_{L^2(\R, L^2)}.
  \end{multline*}
  This proves Step~3.

  \smallskip\noindent\emph{Step 4.}
  $\etahat \in L^2(I, W^{2,2}(M, \Lambda^2)) \cap W^{1,2}(I, L^2(M,
  \Lambda^2))$ and $\p_t \etahat - L'_\rhobar \etahat = 0$.

  As in Lemma~\ref{lem:maximalRegularity} we can see that the operator $\p_t
  + \Ltilde_\rhobar$ has the maximal regularity property.
  Then it follows from this and Step~2 that there exists a weakly converging
  subsequence $\etatilde_\delta$ in $L^2(I, W^{2,2}(M, \Lambda^2)) \cap
  W^{1,2}(I, L^2(M, \Lambda^2))$ that converges to a limit $
  {\etatilde}_*$ as $\delta$ goes to zero. At the same time $\etahat_\delta$
  converges strongly to $\etatilde$ in $L^2(I, \Lambda^2(I, \Lambda^2))$ and
  therefore
  $$
  \etatilde_* = \etatilde \in L^2(I, W^{2,2}(M, \Lambda^2)) \cap
  W^{1,2}(I, L^2(M, \Lambda^2))
  $$
  and
  $$
  \etahat \in L^2(I, W^{2,2}(M, \Lambda^2)) \cap
  W^{1,2}(I, L^2(M, \Lambda^2)).
  $$
  That $\p_t \etahat - L'_\rhobar \etahat = 0$ then follows from
  equation~\eqref{eq:derivativeEtahat}. This proves Step~4.

  \smallskip\noindent\emph{Step 5.}
  $\etahat = 0$.

  By Step~4 we have $\etahat \in C^0(I, W^{1,2}(M, \Lambda^2))$ and therefore
  $$
  \int_M g^{\rhobar_T}(L_\rhobar \phihat + c \phihat, \etahat_T) \dvol = 0
  $$
  for all $\phihat \in C_0^\infty(M, \Lambda)$. It follows that
  $\etahat_T \in W^{1,2}(M, \Lambda^2)$ is a weak solution to the equation
  $$
  L_\rhobar^{*^\rhobar} \etahat_T + c\etahat_T=0.
  $$
  Since $L_\rhobar^{*^\rhobar} + c\,\id$ is injective for $c>c_0$ big enough,
  this implies that $\etahat_T = 0$. But this means that $\etatilde$ is a
  solution to the Cauchy problem
  $$
  \p_t \etatilde + \Ltilde_\rhobar \etatilde= 0, \qquad \etatilde(0,\cdot) = 0.
  $$
  The standart $L^2$ parabolic estimate shows that such a solution is unique
  for $c>0$ big enough and hence $\etatilde = \etahat = 0$.  This proves the
  lemma.
\end{proof}

\begin{lemma}[{\bf Cauchy Problem}]\label{lem:cauchy}
For every $\fhat \in \sYhat$ there exists a unique solution $\rhohat \in
\sXhat$ to the Cauchy problem
\begin{equation}\label{eq:LrhoCauchy}
\p_t \rhohat + L_{\bar{\rho}} \rhohat = \fhat, \qquad
\rhohat(t=0,\cdot) = 0.
\end{equation}
\end{lemma}
\begin{proof}
  Let $\rhohat$ and $\fhat$ be smooth path of smooth exact two-forms with
  support on a compact subset of $I \times \R^4$ satisfying
  $$
  \p_t \rhohat + \triangle \rhohat = \fhat, \qquad \rhohat(t=0, \cdot) =0,
  $$
  where $\triangle$ denotes the standard Laplace operator of the standard
  metric of $\R^4$.
  Then we have from standard parabolic $L^2$-theory the estimate
  $$
  \norm{\rhohat}_{L^2(I, W^{2,2})} + \sup_{t\in I} \norm{\rhohat}_{W^{1,2}} +
  \norm{\p_t \rhohat}_{L^2(I, L^2)} \leq c \norm{f}_{L^2(I, L^2)}.
  $$
  Choosing charts and a subordinate partition of unity for $I\times \R^4$ and
  using the estimate~\eqref{eq:est1ellReg} with $q=2$ shows that there exists
  a polynomial with positive coefficients $\fp(\norm{\rhobar}_{L^\infty(I,
    W^{1,p})})$ such that for all exact $\rhohat \in L^2(I, W^{2,2}) \cap
    W^{1,2}(I, L^2)$ solving~\eqref{eq:LrhoCauchy}
    \begin{multline*}
      \norm{\p_t \rhohat}_{L^2(I,L^2)} + \sup_{t \in
      I}\norm{\rhohat}_{L^2(M,\Lambda^2)} + \norm{\rhohat}_{L^2(I, W^{1,2}(M,
      \Lambda^2))}\\
      \leq
      \fp \left(\norm{\fhat}_{L^2(I, L^2)} +
      \norm{\rhohat}_{L^2(I,L^2)}\right).
    \end{multline*}
  This estimate shows that the operator $\p_t + L_\rhobar + c\,\id$ from the
  exact two-forms in $L^2(I, W^{2,2}) \cap W^{1,2}(I, L^2)$ with initial
  condition $\rhohat(t=0,\cdot) = 0$ to the exact two-forms in $L^2(I, L^2)$
  is injective and has closed image, if the constant $c>0$ is choosen suitably
  big.  If $\rhobar$ is a smooth path of nondegenerate two-forms by
  Lemma~\ref{lem:denseImage} it's image is dense and hence this operator is
  surjective. Therefore for $\rhobar$ smooth
  $$
  \p_t \rhohat + L_\rhobar \rhohat + c \rhohat = \fhat, \qquad
  \rhohat(t=0,\cdot) = 0
  $$
  has a unique exact solution in $L^{2}(I, W^{2,2}) \cap W^{1,2}(I, L^2)$ for
  every exact two-form $\fhat \in L^2(I, L^2)$. From standard parabolic
  regularity theory it follows that this solution is smooth when $\rhobar$ and
  $\fhat$ are smooth. Then $e^{ct}\rhohat$ is the unique smooth solution
  to~\eqref{eq:LrhoCauchy}.

  We argue exactly as in the maximal regularity
  Lemma~\ref{lem:maximalRegularity} for the operator $d^{*^\rho}\frac{d}{u}$
  that there exists a constant $c = c(\norm{\rhobar}_{L^\infty(I, W^{1,p})}) >
  0$ such that for all exact $\rhohat\in C^\infty(I, C^\infty(M, \Lambda^2))$,
  \begin{equation}\label{eq:maxRegEst}
    \norm{\p_t \rhohat }_{L^2(I, L^p)} \leq c (\norm{\p_t \rhohat +
      L_{\rhobar}\rhohat}_{L^2(I, L^p)} + \norm{\rhohat}_{L^2(I, L^p)}).
  \end{equation}
  The assumption~\eqref{eq:maxRegAssumption} can easily be seen to be
  satisfied for $p>4$. The additional lower order terms introduced by
  $A^\rhobar$ can be controlled by the Gagliardo-Nirenberg interpolation
  inequality. To find a solution to~\eqref{eq:LrhoCauchy} in $\sXhat$ we
  choose a sequence of smooth paths of smooth exact two-forms
  $\{\fhat_k\}_{k\in \N} $ and a sequence of smooth paths of smooth nondegenerate
  two-forms $\{\rhobar_k\}_{k\in \N}$ such that
  $$
  \lim_{k\to \infty}\norm{\fhat - \fhat_k}_{L^2(I, L^p)}=0, \quad
  \lim_{k\to\infty} \norm{\rhobar - \rhobar_k}_{C^0(I, W^{1,p})} =0.
  $$
  Then for every $k\geq 0$ there exists a smooth exact solution $\rhohat_k$ to
  $$
  \p_t \rhohat_k + L_{\rhobar_k} \rhohat_k = \fhat_k, \qquad
  \rhohat_k(t=0,\cdot) = 0.
  $$
  The maximal regularity estimate~\eqref{eq:maxRegEst} holds with a uniform
  constant in a neighbourhood of $\rhobar$. Therefore there exists a weakly
  converging subsequence converging to a solution $\rhohat \in \sX$
  of~\eqref{eq:LrhoCauchy}. This proves the lemma.
\end{proof}
\begin{proof}[Proof of Theorem~\ref{thm:semiflow}]\label{proof:semiflow}
We need to show that the Donaldson flow depends smoothly on the initial
conditions. Let $\rho_0 \in \sS^{1,p}_2(M ,\Lambda^2)$ and $\bar{\rho}$ be
the solution to the Donaldson flow in $W^{1,2,p}(M_I, \Lambda^2)$ for the
initial condition $\rho_0 \in \sS^{1,p}_2(M, \Lambda^2)$ on the time
interval $I= [0,T]$. Let us define the following spaces.
\begin{align*}
  \sX_0 := \big\{ \rho \in & L^2(I, W^{2,p}(M, \Lambda^2)) \cap W^{1,2}(I,
  L^p(M, \Lambda^2))\\
  &|\ \forall t: d\rho(t,\cdot) = 0,
  \ \rho(t,\cdot) \wedge \rho(t,\cdot) > 0,
  \ [\rho(t,\cdot)] =a,\\
&\ \rho(t=0,\cdot) = \rho_0\big\}.
\end{align*}
This is the space of symplectic forms in $W^{1,2,p}(M_I, \Lambda^2)$
representing the fixed cohomology class $a\in H^2(M; \R)$ and have the fixed
initial condition $\rho_0$. The formal tangent space to $\sX_0$ at $\rho$ is
given by
\begin{align*}
\sXhat_0 := \{\rhohat \in &L^2(I, W^{2,p}(M, \Lambda^2)) \cap W^{1,2}(I,
L^p(M, \Lambda^2))\\
&|\ \forall t: \rhohat(t,\cdot)\ \text{exact}, \ \rhohat(t=0, \cdot) =0\}.
\end{align*}
The space $\sS^{1,p}_{2,a}(M, \Lambda^2)$ is an open subset of an affine
space, the corresponding vector space is
$$
\sZhat := \{\tauhat \in B^{1,p}_2(M, \Lambda^2) | \ \tauhat \
\text{exact} \}.
$$
There exists a bounded linear extension operator
$$
\sT: \sZhat \to \sXhat
$$
such that
$$
(\sT \tauhat) (0) = \tauhat.
$$
For example we can define $\sT$ as follows. Let $T(t)$ be the semigroup
created by the negative Hodge laplacian with respect to the background
metric $g$ on the exact two-forms in $L^p(M, \Lambda^2)$ with domain the
exact two-forms in $W^{2,p}(M, \Lambda^2)$. The theory developed
in~\cite{Grigoryan} asserts that $t \mapsto T(t)\tauhat$ is a continous map
from $I$ to $\sZhat$.  We define
$$
(\sT\tauhat)(t) := T(t) \tauhat.
$$
Now consider the map
$$
\sF : \sX_0 \times \sZhat \to \sYhat,\qquad \sF(\rho, \tauhat) := \p_t \rho +
\p_t (\sT \tauhat) - d *^{\rho + \sT \tauhat} d \theta^{\rho + \sT
\tauhat}.
$$
Then,
$$
\sF(\bar{\rho}, 0) = 0
$$
and this map is clearly infinitely Fr\'echet differentiable. We compute the
linearization in the first factor at $(\bar{\rho}, 0)$,
$$
d \sF(\bar{\rho}, 0) : \sXhat \to \sYhat, \qquad d\sF (\bar{\rho},0)
\rhohat:= \left.\frac{d}{ds}\right|_{s=0} \sF(\bar{\rho} + s\rhohat,
\tauhat) = \p_t \rhohat + L_{\bar{\rho}} \rhohat,
$$
where $L_{\bar{\rho}}$ is given by equation~\eqref{eq:Lrho}.
By Lemma~\ref{lem:cauchy} the Cauchy problem
$$
\p_t \rhohat + L_{\bar{\rho}} \rhohat = \fhat, \qquad \rhohat(t=0,\cdot) = 0
$$
has a unique solution $\rhohat \in \sXhat$ for every $\fhat \in \sYhat$.
Therefore $d \sF(\bar{\rho}, 0)$ is bijectiv. It follows from the implicit
function theorem for Banach spaces that there exists an open neighborhood $0
\in \sU \subseteq \sZhat$ and a smooth map
$$
\Phi : \sU \to \sX_0
$$
with
$$
\sF(\Phi(\tauhat),\tauhat) = 0, \qquad \Phi(0) = \bar{\rho}.
$$
Therefore,
$$
\p_t(\Phi(\tauhat) + \sT \tauhat) = d *^{\Phi(\tauhat) + \sT \tauhat} d
\theta^{\Phi(\tauhat) + \sT \tauhat}, \qquad (\Phi(\tauhat) + \sT
\tauhat)(t=0,\cdot) = \rho_0 + \tauhat.
$$
In particular the map
$$
\tauhat \mapsto \Phi(\tauhat) + \sT\tauhat
$$
is a smooth map from initial conditions in $\sU$ to solutions of the
Donaldson flow in $W^{1,2,p}(M_I, \Lambda^2)$ with these initial conditions.
This proves the theorem.
\end{proof}

The semiflow property of the Donaldson flow allows one to apply classic
stability analysis results of dynamical systems to the time one map of the
Donaldson flow. In particular we can prove that there is a neighborhood around
the absolute minimum in $\sS_a$ in the Besov space topology which is a local
stable manifold for the absolute minimum. Every solution to the Donaldson flow
whose initial conditions lay within that neighborhood converges to the
absolute minimum.

\begin{corollary}
There exists an open set in the topology of the Besov space $B^{1,p}_2(M,
\Lambda^2)$ which is a local stable manifold around the absolute minimum of
$\sS_a$, i.e. the iterates of the time one map of the Donaldson flow
converge to the absolute minimum for every point in this neighborhood.
\end{corollary}
\begin{proof}
By Theorem~\ref{thm:semiflow} there exists an open set around the absolute
minimum $\om$ in $\sS^{1,p}_{2,a}$ such that every solution to the Donaldson
flow with initial condition in this neighborhood exists on the interval
$[0,1]$. Let $\phi^1: B_\om \to \sS^{1,p}_{2,a}$ be the time one map of the
Donaldson flow restricted to a small ball centered at $\om$ within that
neighborhood. We denote
$$
\sZhat := T_\om \sS^{1,p}_{2,a} = \{\tauhat \in B^{1,p}_2(M, \Lambda^2)|\
\tauhat \text{ exact} \}.
$$
Let $\phihat : \sZhat \to \sZhat$ be the linearization of $\phi^1$ at $\om$.
It is given by the solution to the linearized Donaldson flow equation,
\begin{equation}\label{eq:negHeatFlow}
  \p_t \rhohat + L_\om \rhohat = 0, \qquad \rhohat(t=0, \cdot) = \tauhat
  \in \sZhat,
\end{equation}
where $\om \in \sS_a$ is the unique minimum of the energy functional
$$
\sE\left( \rho \right) = \int_M \frac{2 \Abs{\rho^+}^2}{\Abs{\rho^+}^2 -
\Abs{\rho^-}^2}\dvol, \qquad \rho \in \sS_a
$$
and $L_\om$ is given by~\eqref{eq:Lrho}. Recall that $\om$ is a symplectic
form on $M$ compatible with the background metric and $[w] = a \in H^2(M; \R)$.
In particular $\om$ is the unique self-dual symplectic form on $M$
representing the cohomology class $a$.  Since $\om$ is compatible with the
background metric, $*^\om$ coincides with the Hodge star operator of the
background metric, $u_\om = \frac{\dvol_\om}{\dvol}= 1$ and
$$
\theta^\om = \om - \frac{1}{2}\Abs{\om}^2\om = \om - \om = 0.
$$
It follows that the linearized gradient operator $L_\om$ is the
Hodge laplacian,
$$
L_\om = d \frac{d^{*^\om}}{u_\om} + A^\om = d d^{*}.
$$
Thus the solution to~\eqref{eq:negHeatFlow} is the heat flow on $\sZhat$.
This is well know to be a contraction and
$$
\norm{\phihat}_{L(\sZhat,\sZhat)} < 1,
$$
where $\norm{\cdot}_{L{(\sZhat, \sZhat)}}$ denotes the operator norm for
operators on $\sZhat$.  The map $\phi^1$ is $C^\infty$-Fr\'echet
differentiable and from the previous inequality we know that it's derivative
is a contraction on an open ball around $\om$ in $\sS^{1,p}_{2,a}$. It follows
from the mean value theorem that $\phi^1$ is a contraction on this ball.  The
Banach fixed point theorem now implies that the iteration of $\phi^1$ on this
ball converges to a unique fixed point and this clearly is $\om$. This proves
the theorem.
\end{proof}
\begin{remark}
Alternatively, we can prove the existence of an open neighborhood around the
absolute minimum of the energy functional in $\sS^{1,p}_{2,a}(M, \Lambda^2)$
such that the Donaldson flow converges exponentially fast to the absolute
minimum for every initial condition laying in this neighborhood by repeating
the proof of the short time existence theorem~\ref{thm:localExistence},
where one applies the Banach fixed point theorem to the analogous iteration
on a small ball around zero in the space
\begin{align*}
  \sX_{\delta,\epsilon,\rho_0} := \{ \rhohat \in W^{1,2,p}(M_{\R_+},
  \Lambda^2) |\ &\rhohat \text{ exact},\\
  &\rhohat(t=0,\cdot) = 0,\\
  &\norm{e^{\delta t}(\rhotilde_0 +
\rhohat - \om}_{W^{1,2,p}} < \epsilon\}.
\end{align*}
Here $\rhotilde_0$ is a extension of the initial condition $\rho_0 \in
\sS^{1,p}_2$ to $W^{1,2,p}(M, \Lambda^2)$, $\om$ is the absolute minimum of
the energy functional and $\delta,\epsilon >0$ are real constants that have
to be choosen sufficiently small. The norm $\norm{\cdot}_{W^{1,2,p}}$ is
defined by~\eqref{def:norm}.
\end{remark}


\appendix
\section{Products in Sobolev Spaces.}\label{app:sobolev}
A proper open connected subset $\Om \subset \R^n$ is called a {\it smooth
domain}\ if for every $x\in \p \Om$ there is a ball $B = B(x)$ and a smooth
diffeomorphism $\psi$ of $B $ onto $D \subset \R^n$ such that
\begin{gather*}
\psi (B \cap \Om) \subset \R^n_+, \qquad \psi(B \cap \p \Om) \subset \p
\R^n_+.
\end{gather*}
If $\psi$ is not smooth but of class $C^k$ then the domain is said to have $C^k$
boundary.
\begin{proposition}[{\bf Gagliardo-Nirenberg}] \label{prop:gagliardo}
Let $\Om\subset \R^n$ be a bounded open domain with $C^k$ boundary. Suppose
that $j,k \geq 0$ are integers with $j<k$ and $1 \leq p,q,r \leq \infty$
with $k-n/q + n/r \geq 0$ and
$$
j - \frac{n}{p} = \lambda \left(k - \frac{n}{q}\right) + (1 - \lambda)
\left(- \frac{n}{r}\right), \qquad \frac{j}{k} \leq \lambda \leq 1.
$$
If $(k-j)q = n$ assume also that $\lambda \neq 1$. Then there exists a
constant $c>0$ such that
$$
\norm{\p^j u}_{L^p}\leq c \norm{\p^k u}_{L^q}^\lambda
\norm{u}_{L^r}^{1-\lambda}
$$
for $u \in W^{k,q}(\Om)$.
\end{proposition}
\begin{proof} A proof of this proposition can be found in \cite{FRIEDMAN}.
\end{proof}
\begin{lemma}[{\bf Products in Sobolev Spaces}]\label{lem:sobolevProd}
\smallskip\noindent{\bf (i)}
Let $M$ be a closed manifold of dimension $n$. Let $f,g \in C^\infty(M,
\R)$.  Let $k, p \in \N$ such that $k - \frac{n}{p} > 0$. There exists a
constant $c>0$ such that
$$
\norm{f g}_{W^{k,p}} \leq c \left(\norm{f}_{L^\infty} \norm{g}_{W^{k,p}} +
\norm{f}_{W^{k,p}} \norm{g}_{L^\infty}\right).
$$

\smallskip\noindent{\bf (ii)}
Let $E \to M$ be a smooth vector bundle over $M$. Let $k - \frac{n}{p} > 0$.
Let $f \in C^\infty(E,\R)$ and $u \in C^\infty(M, E)$. Let $U\subset E$ be
an open set containing the image of $u$. Then there exists a constant $c>0$
such that
\begin{equation*}
\norm{f \circ u}_{W^{k,p}}
\leq
c \Abs{f}_{C^k(U)} \left(1 + \norm{u}_{W^{k,p}}\big(1 +
\norm{u}_{L^\infty}\big)^{k-1}\right).
\end{equation*}
\end{lemma}
\begin{proof}
It is enough to prove these statements over an open domain of $\R^n$ with
smooth boundary. We prove (i). Let $f, g \in C^\infty(\Om, \R)$. Let $0 \leq
i\leq k$ and let $\p^i = \p^\alpha$ for an arbitrary multi-indice $\alpha
\in \N^n$ with $\Abs{\alpha} = i$.  By the product rule
$$
\p^i (fg) = \sum_{\ell=0}^i (\p^\ell f) (\p^{i - \ell}g) = f (\p^i g) +
(\p^i f) g + \sum_{\ell = 1}^{i-1} (\p^\ell f) (\p^{i - \ell}g).
$$
Clearly,
$$
\norm{f (\p^i g) +(\p^i f) g }_{L^p} \leq
\norm{f}_{L^\infty}\norm{g}_{W^{k,p}} +\norm{f}_{W^{k,p}}\norm{g}_{L^\infty}.
$$
Let $s = \frac{pk }{\ell}, t = \frac{pk}{k-\ell}$. Then by the H\"older
inequality and the Gagliardo-Nirenberg interpolation
inequality~\ref{prop:gagliardo}
\begin{equation*}
\begin{split}
  \norm{(\p^\ell f) (\p^{i - \ell}g)}_{L^p}
  &\leq
  \norm{\p^\ell f}_{L^s}\norm{\p^{i - \ell}g}_{L^t}\\
  &\leq
  c_1 \norm{f}^{\lambda_1}_{W^{k,p}}\norm{f}_{L^\infty}^{1-\lambda_1}
  \norm{g}_{W^{k,p}}^{\lambda_2} \norm{g}_{L^\infty}^{1- \lambda_2}
\end{split}
\end{equation*}
for
\begin{equation*}
\lambda_1 = \frac{\ell - \frac{n}{s}}{k - \frac{n}{p}} = \frac{\ell}{k},
\qquad \lambda_2 = \frac{k - \ell - \frac{n}{t}}{k-\frac{n}{p}} = \frac{k
- \ell} {k},\qquad 1 \leq \ell \leq i - 1.
\end{equation*}
By Young's inequality and $\lambda_1 + \lambda_2 = 1$
\begin{equation*}
\begin{split}
  \norm{(\p^\ell f) (\p^{i - \ell}g)}_{L^p}
  &\leq
  c_1 \norm{f}^{\lambda_1}_{W^{k,p}}
  \norm{g}_{L^\infty}^{\lambda_1}\norm{f}_{L^\infty}^{\lambda_2}
  \norm{g}_{W^{k,p}}^{\lambda_2}\\
  &\leq
  c_2 (\norm{f}_{L^\infty}
  \norm{g}_{W^{k,p}} + \norm{f}_{W^{k,p}}
  \norm{g}_{L^\infty}).
\end{split}
\end{equation*}
This proves (i).

We prove (ii).
Let $U\subset E$ be an open set such that $u(M) \subseteq U$. Clearly
$$
\norm{f\circ u}_{L^p} \leq \left({\Vol(M)}\right)^{\frac{1}{p}}
\Abs{f}_{C^0(U)}.
$$
Let $\p^i = \p^\alpha$ for an arbitrary multi-indice $\alpha \in \N^n$ with
$\Abs{\alpha} = i$. We claim that $\p^i f(u)$ equals the sum of terms of the
form
\begin{equation}\label{eq:productForm}
\p^{\ell_0} f(u)\cdot (\p^{1 }u)^{\ell_1} \cdot
(\p^{2}u)^{\ell_2}
\cdots (\p^{i}u)^{\ell_i}
\end{equation}
for $i \geq 1$, $1\leq \ell_0 \leq i$ and
$\sum_{j = 1}^i j \ell_j =
i$. By the chain rule $ \p f(u) = \p f(u) \p u$ and hence the claim holds true
for $i=1$. Suppose the claim is true for an $i \geq 1$. Then
\begin{equation*}
\begin{split}
  \p^{i + 1} (f (u))
  &= \p \p^{i} (f(u)) \\
  &= \p \left(\p^{\ell_0} f (u) \left((\p^{1 }u)^{\ell_1}
  (\p^{2}u)^{\ell_2}
  \cdots (\p^{i}u)^{\ell_i}\right)\right)\\
  &=
  \p^{\ell_0 + 1 } f(u) \left((\p^{1 }u)^{\ell_1 + 1}
  (\p^{2}u)^{\ell_2}
  \cdots (\p^{i}u)^{\ell_i}\right) \\
  &\qquad
  + \p^{\ell_0}f (u) \p \left((\p^{1 }u)^{\ell_1}
  (\p^{2}u)^{\ell_2}
  \cdots (\p^{i}u)^{\ell_i}\right)
\end{split}
\end{equation*}
and the summands of the last expression are again of the claimed form. This
proves the claim. By the H\"older inequality
\begin{multline*}
\norm{\p^{\ell_0} f(u) (\p^{1 }u)^{\ell_1}
(\p^{2}u)^{\ell_2} \cdots (\p^i u)^{\ell_i}}_{L^p} \\
\leq \Abs{f}_{C^{\ell_0}(U)}\norm{\p^1 u}_{L^{s_1}}^{\ell_1}\norm{
\p^2 u}_{L^{s_2}}^{\ell_2} \cdots \norm{\p^i u}_{L^{s_i}}^{\ell_i}
\end{multline*}
for $s_j = \frac{p \cdot i} {j}$. Let
$$
\lambda_j = \frac{j - \frac{n}{s_j} }{i - \frac{n}{p}} = \frac{j}{i}.
$$
By the Gagliardo-Nirenberg interpolation
inequality
\begin{equation*}
\begin{split}
  \norm{\p^1 u}_{L^{s_1}}^{\ell_1}\norm{
  \p^2 u}_{L^{s_2}}^{\ell_2} \cdots \norm{\p^i u}_{L^{s_i}}^{\ell_i}
  &\leq
  c_3
  \norm{u}_{W^{i, p}}^{\ell_1 \lambda_1}
  \norm{u}_{L^\infty}^{\ell_1 (1- \lambda_1)} \cdots \norm{u}_{W^{i,
  p}}^{\ell_i \lambda_i} \norm{u}_{L^\infty}^{\ell_i (1 - \lambda_i)}\\
  &\leq
  c_3 \norm{u}_{W^{k,p}} \norm{u}_{L^\infty}^{(\sum_{j=1}^i \ell_j) - 1}.
\end{split}
\end{equation*}
Since $\sum_{j = 1}^ i \ell_j \leq i$, $i \leq k$ and
$\ell_0 \leq k$  this shows that
\begin{equation*}
\begin{split}
  \norm{f \circ u}_{W^{k,p}}
  &\leq
  c_4 \Abs{f}_{C^k(U)} \left(1 + \norm{u}_{W^{k,p}} +
  \norm{u}_{W^{k,p}}\norm{u}_{L^\infty} + \norm{u}_{W^{k,p}}
  \norm{u}_{L^\infty}^2 +\right.\\
  &\left. \cdots +
  \norm{u}_{W^{k,p}}\norm{u}_{L^\infty}^{k-1}\right)
\end{split}
\end{equation*}
for a constant $c_4 >0$. This proves (ii).
\end{proof}
Recall that we defined
$$
\norm{u}_{W^{r, k,p}(M_I)} := \sum_{\substack{2s + \ell \leq
k \\ s \leq r}} \norm{\p_t^s
\p^\ell u}_{L^{p,2}\left( M_I \right)}.
$$
for a function $u\in C^\infty(M_I,\R)$ and an open interval $I\subseteq \R$.
\begin{corollary}\label{cor:parabolicProduct}
Let $1 - \frac{n}{p} > 0$, $p\geq 2$, $k\geq 2$, $s\geq 1$.  Let $u,v \in
C^\infty(M_I, \R)$ for an open interval $I\subseteq \R$. Then
\begin{equation*}
\norm{u v }_{W^{s,k,p}}
\leq \norm{u}_{W^{s,k,p}}\norm{v}_{W^{s,k,p}} \end{equation*}
\end{corollary}
\begin{proof}
For all $1 \leq r \leq s$ and a fixed $t\in I$
\begin{multline*}
\norm{\p_t^r (uv)}_{W^{k-2r,p}}\\
\leq
\norm{(\p_t^r u) v }_{W^{k-2r,p}} + \norm{u \p_t^r v}_{W^{k-2r,p}}
+
\sum_{i=1}^{r-1} \norm{(\p_t^i u) \p_t^{r- i} v}_{W^{k-2r,p}}.
\end{multline*}
If $k-2r=0$, then
\begin{equation*}
\begin{split}
  &\norm{(\p_t^r u) v }_{L^2(I, L^p)} + \norm{u \p_t^r v}_{L^2(I, L^p)}\\
  &\qquad\leq
  \norm{\p_t^ru}_{L^2(I, L^p)}\norm{v}_{L^\infty(I, L^\infty)} +
  \norm{u}_{L^\infty(I, L^\infty)}\norm{\p_t^r v}_{L^2(I, L^p)}\\
  &\qquad\leq
  \norm{u}_{W^{s,k,p}}\norm{v}_{W^{s,k,p}}.
\end{split}
\end{equation*}
and for $1 \leq i \leq r -1 $
\begin{equation*}
\norm{(\p_t^i u) \p_t^{r- i} v}_{L^2(I, L^p)} =
\norm{\p_t^i u}_{L^\infty(I, L^\infty)} \norm{\p_t^{r- i}
v}_{L^2(I, L^p)}
\leq
\norm{u}_{W^{s,k,p}}\norm{v}_{W^{s,k,p}}.
\end{equation*}
Here we use that for $p\geq 2$ the identity operator on real smooth
functions on $I\times M$ with compact support extends to a bounded linear
operator
$$
L^2(I, W^{2,p}) \cap W^{1,2}(I, L^p) \to C(I, W^{1,p})
$$
and hence for $1 - \frac{n}{p}>0$ and $s\geq 1$
$$
\norm{u}_{L^\infty(I, L^\infty)} \leq \norm{u}_{L^\infty(I, W^{1,p})} \leq \norm{u}_{W^{s,k,p}}.
$$
and likewise for $1\leq i \leq r-1$
$$
\norm{\p_t^i u}_{L^\infty(I, L^\infty)} \leq \norm{\p_t^i u}_{L^\infty(I, W^{1,p})} \leq
\norm{u}_{W^{s,k,p}}.
$$
If $1 \leq k - 2r < k$ then
\begin{equation*}
\begin{split}
&\norm{(\p_t^r u) v }_{W^{k-2r,p}} + \norm{u \p_t^r v}_{W^{k-2r,p}}\\
&\qquad\qquad\qquad\qquad\leq
\norm{\p_t^ru}_{W^{k-2r,p}}\norm{v}_{L^\infty} +
\norm{\p_t^ru}_{L^\infty}\norm{v}_{W^{k-2r,p}}\\
&\qquad\qquad\qquad\qquad \qquad+
\norm{u}_{W^{k-2r,p}}\norm{\p_t^r v}_{L^\infty} +
\norm{u}_{L^\infty}\norm{\p_t^r v}_{W^{k-2r,p}}\\
\end{split}
\end{equation*}
and
\begin{equation*}
\begin{split}
&\norm{(\p_t^r u) v }_{L^2(I,W^{k-2r,p})} + \norm{u \p_t^r
v}_{L^2(I,W^{k-2r,p})}\\
&\qquad\leq
\norm{\p_t^ru}_{L^2(I,W^{k-2r,p})}\norm{v}_{L^\infty(I,L^\infty)} +
\norm{\p_t^ru}_{L^2(I, W^{1,p})}\norm{v}_{L^\infty(I,W^{k-2r,p})}\\
&\qquad \qquad+
\norm{u}_{L^\infty(I,W^{k-2r,p})}\norm{\p_t^r v}_{L^2(I, W^{1,p})} +
\norm{u}_{L^\infty(I,L^\infty)}\norm{\p_t^r v}_{L^2(I,W^{k-2r,p})}\\
&\qquad\leq
\norm{u}_{W^{s,k,p}}\norm{v}_{W^{s,k,p}}.
\end{split}
\end{equation*}
Further, for $1\leq i \leq r -1$
\begin{equation*}
\norm{(\p_t^i u) \p_t^{r- i} v}_{W^{k-2r,p}} \leq \norm{\p_t^i
u}_{L^\infty} \norm{\p_t^{r- i} v}_{W^{k-2r,p}} + \norm{\p_t^i
u}_{W^{k-2r,p}} \norm{\p_t^{r- i} v}_{L^\infty}
\end{equation*}
and
\begin{equation*}
\begin{split}
&\norm{(\p_t^i u) \p_t^{r- i} v}_{L^2(I,W^{k-2r,p})}\\
&\quad\leq
\norm{\p_t^i u}_{L^\infty(I, W^{1,p})} \norm{\p_t^{r- i}
v}_{L^2(I,W^{k-2r,p})} + \norm{\p_t^i u}_{L^2(I,W^{k-2r,p})} \norm{\p_t^{r-
i} v}_{L^\infty(I,W^{1,p})}\\
&\quad\leq
\norm{u}_{W^{s,k,p}}\norm{v}_{W^{s,k,p}}.
\end{split}
\end{equation*}
Finally, if $r=0$, then
$$
\norm{uv}_{W^{k,p}} \leq \norm{u}_{L^\infty}\norm{v}_{W^{k,p}} +
\norm{u}_{W^{k,p}}\norm{v}_{L^\infty}
$$
and
\begin{equation*}
\begin{split}
\norm{uv}_{L^2(I,W^{k,p})}
&\leq
\norm{u}_{L^\infty(I,L^\infty)}\norm{v}_{L^2(I,W^{k,p})} +
\norm{u}_{L^2(I,W^{k,p})}\norm{v}_{L^\infty(I,L^\infty)}\\
&\leq
\norm{u}_{W^{s,k,p}}\norm{v}_{W^{s,k,p}}.
\end{split}
\end{equation*}
This proves the corollary.
\end{proof}



\begin{thebibliography}{99}
\scriptsize

\vspace{-7pt}

\bibitem{ANGENENT}
Sigurd B. Angenent,
Nonlinear analytic semiflows.
{\it Proceedings of the Royal Society of Edingburg}
{\bf 115A} (1990), 91--107

\vspace{-7pt}

\bibitem{DON1}
S.K.~Donaldson,
Moment Maps and Diffeomorphisms.
{\it Asian J.\ Math.} {\bf 3} (1999), 1--16. \\
\url{http://bogomolov-lab.ru/G-sem/AJM-3-1-001-016.pdf}



\vspace{-7pt}

\bibitem{DON2}
S.K.~Donaldson,
Two-Forms on Four-Manifolds and Elliptic Equations.
{\it Inspired by S. S. Chern}, edited by Phillip A.~Griffiths,
Nankai Tracts Mathematics {\bf 11},
World Scientific, 2006, pages 153--172.



\vspace{-7pt}

\bibitem{FRIEDMAN}
A.~Friedman,
Partial Differential Equations.
{\it Robert E. Krieger Publishing Company (1976)}.



\vspace{-7pt}

\bibitem{Grigoryan}
A.~Grigor'yan, L.~Liu, Heat kernel and Lipschitz-Besov spaces. Preprint, April
2014. {\it Forum Math.} (2014), DOI:10.1515/forum-2014-0034.
\url{https://www.math.uni-bielefeld.de/~grigor/besov.pdf}

\vspace{-7pt}

\bibitem{KROMSAL}
R.S.~Krom, D.A.~Salamon, The Donaldson geometric flow for symplectic
four-manifolds, ETH Z\"urich.

\vspace{-7pt}

\bibitem{Lamberton}
L.~Lamberton, \'Equations d'\'evolution lin\'eaires associ\'ees \`a des
semi-groupes de contraction dans les espace $L^p$. {\it Journal of Functional
Analysis {\bf 72} (1987), 252--262}

\vspace{-7pt}






\bibitem{ROBSAL}
J.~Robbins, D.A.~Salamon, The Spectral Flow and the Maslov Index.
{\it Bulletin of the London Mathematical Society} {\bf 27.1} (1995), 1--33.
\url{http://www.math.ethz.ch/~salamon/PREPRINTS/spec.pdf}

\vspace{-7pt}

\bibitem{SAL}
D.A.~Salamon, Uniqueness of symplectic structures.
{\it Acta Mathematica Vietnamica} {\bf 38} (2013), 123--144.
\url{http://www.math.ethz.ch/~salamon/PREPRINTS/unique.pdf}






\end{thebibliography}
\end{document}